\documentclass[
a4paper
,11pt
]{amsart}

\usepackage{amssymb,latexsym}
\usepackage[latin9]{inputenc}
\usepackage{amsmath}
\usepackage{amsfonts}
\usepackage{dsfont}
\usepackage{amsthm}
\usepackage{thmtools}

\declaretheorem[parent=section]{theorem}
\declaretheorem[sibling=theorem]{proposition}
\declaretheorem[sibling=theorem]{lemma}
\declaretheorem[sibling=theorem]{corollary}
\declaretheorem[sibling=theorem]{conjecture}

\declaretheorem[sibling=theorem,style=remark]{remark}

\declaretheorem[sibling=theorem,style=definition]{definition}

\declaretheorem[numbered=no,name=Theorem]{theorem*}
\declaretheorem[numbered=no,style=definition,name=Definition]{definition*}
\declaretheorem[numbered=no,name=Proposition]{proposition*}

\usepackage{enumerate}
\usepackage{ifthen}
\usepackage{comment}
\usepackage{xcolor}
\usepackage{mathtools}
\usepackage{booktabs}
\usepackage{etoolbox}
\usepackage[T1]{fontenc}

\usepackage{hyperref}
\hypersetup{
   backref
   ,colorlinks=true
   ,allcolors=black
}

\begin{document}
\bibliographystyle{amsalpha}

\newcommand{\mi}{\ensuremath{\mathrm{i}}}
\newcommand{\me}{\ensuremath{\mathrm{e}}}
\newcommand{\mPS}{\ensuremath{(\Omega,\mathcal{F},P)}}
\newcommand{\Id}{\on{Id}}
\newcommand{\probSpace}{\ensuremath{(\Omega,\mathcal{F},P)}}
\newcommand{\fprobSpace}{\ensuremath{(\Omega,\mathcal{F}_t,P)}}
\newcommand{\prob}{\ensuremath{\operatorname{P}}}
\newcommand{\fps}{\ensuremath{\left( \Omega,\mathcal{F},(\mathcal{F}_t)_{t
        \geq 0}, \prob \right)}}
\newcommand{\stset}{\ensuremath{\mathcal{T}}}
\newcommand{\ind}{1{\hskip -2.5 pt}\hbox{\textnormal{I}}}

\newcommand{\on}[1]{\operatorname{#1}}
\newcommand{\abs}[1]{\left\vert #1 \right\rvert}
\newcommand{\norm}[1]{\left\lVert #1 \right\rVert}
\newcommand{\Ex}[1]{\operatorname{E} \left[ #1 \right]}
\newcommand{\qvar}[2]{\langle #1, #2 \rangle}
\newcommand{\Set}[1]{\left\{ #1 \right\}}
\newcommand{\diff}[1]{\operatorname{d}\ifthenelse{\equal{#1}{}}{\,}{#1}}

\newcommand{\E}{\operatorname{E}}
\newcommand{\C}{\mathbb{C}}
\newcommand{\R}{\mathbb{R}}
\newcommand{\Q}{\mathbb{Q}}
\newcommand{\N}{\mathbb{N}}
\newcommand{\Z}{\mathbb{Z}}
\newcommand{\LL}{\on{L}}

\newcommand{\Cum}{\mathcal{K}}

\newtheorem*{ld_thm:mix}{Theorem \ref{ld_thm:mix}}

\author{Simon Campese}
\email{simon@campese.de}
\address{Universit\`a degli Studi di Roma Tor Vergata, Dipartimento di Matematica, Via della
  Ricerca Scientifica 1, 00133 Roma, Italy}
\title[Moments in space of the increments of Brownian local time]{A limit theorem
  for moments in space of the increments of Brownian local time}
\date{\today}
\thanks{Research partially supported by ERC grant 277742 Pascal}  
\begin{abstract}
  We proof a limit theorem for moments in space of the increments of Brownian
  local time. As special cases for the second and third moments, previous
  results by Chen et al. (Ann. Prob. 38, 2010, no. 1) and Rosen~(Stoch. Dyn. 11, 2011, no. 1),
  which were later reproven by Hu and Nualart~(Electron. Commun. Probab. 14,
  2009; Electron. Commun. Probab. 15, 2010) and
  Rosen~(S\'eminaire de Probabilit\'es XLIII, Springer, 2011) are
  included. Furthermore, a conjecture of Rosen for the fourth moment is
  settled. 
  In comparison to the previous methods of proof, we follow a fundamentally
  different approach by exclusively working in the space variable of the
  Brownian local time, which allows to give a unified argument for arbitrary
  orders. The main ingredients are Perkins' semimartingale
  decomposition, the Kailath-Segall identity and an asymptotic Ray-Knight
  Theorem by Pitman and Yor.
\end{abstract}
\maketitle
\section{Introduction}

Let $(L_t^x)$ be the local time of Brownian motion. 
In~\cite{chen_clt_2010}, motivated by the form of a
Hamiltonian in a certain polymer model, Chen et al. proved that 
\begin{equation}
  \label{eq:9}
  \frac{1}{h^{3/2}} \left(
  \int_{-\infty}^{\infty} \left( L_t^{x+h} - L_t^x \right)^2 \diff{x} -4ht
\right)
\xrightarrow{d} \sqrt{\frac{64}{3} \int_{-\infty}^{\infty} (L_t^x)^2 \diff{x}}
  \, Z,
\end{equation}
where $Z$ is a standard Gaussian random variable which is independent of
$(L_t^x)_{x \in \R}$ and $\xrightarrow{d}$ denotes convergence in
distribution. This can be seen as a central limit theorem, as it was 
also shown in the aforementioned reference that 
\begin{equation*}
  \E \left[ h^{-3/2} \int_{-\infty}^{\infty} \left(
      L_t^{x+h}-L_t^x \right)^2\diff{x}  \right]
  =
  4th + \mathcal{O}(h^2).
\end{equation*}
Note that up to a constant, the limit on
the right hand side equals in law $\int_{-\infty}^{\infty} L_t^x
  \diff{W_x}$, where $W$ is a two-sided Brownian 
motion which is independent of $(L_t^x)_{x \in \R}$. As was pointed out
in~\cite{chen_clt_2010}, this integral, when interpreted as a process in $t$, is
known as ``Brownian motion in Brownian scenery'' and also appears as a limit in
several applications, for 
example when modelling self-interacting random walks
(see~\cite{kesten_limit_1979}) or charged polymers
(see~\cite{chen_limit_2008,chen_charged_2009}). 

The proof of~\eqref{eq:9} was realised by the method of moments, and two
different ones were subsequently given by Rosen in~\cite{rosen_stochastic_2011},
using stochastic calculus and Brownian self-intersection local times, and
Hu and Nualart in~\cite{hu_stochastic_2009}, using Malliavin calculus and Pitman
and Yor's asymptotic version of the Ray-Knight theorem
(see~\cite{pitman_asymptotic_1986}).
Later, in~\cite{rosen_clt_2011}, Rosen, again using the method of moments,
proved a central limit theorem for the third power, which reads
\begin{equation}
  \label{eq:32}
  \frac{1}{h^2} \int_{-\infty}^{\infty} \left( L_t^{x+h} - L_t^x \right)^3
  \diff{x} \xrightarrow{d} \sqrt{192 \int_{-\infty}^{\infty} \left( L_t^x
    \right)^3 \diff{x}} \, Z,
\end{equation}
with $Z$ as above. For this case as well, a Malliavin calculus proof was
provided by Hu and Nualart in~\cite{hu_central_2010}. Unfortunately, as
Rosen mentions in~\cite{rosen_clt_2011}, his proof via the method of moments
will not work for powers higher than three. It
yields, however, the following conjecture\footnote{The rightmost integral in the
  numerator of~\eqref{eq:47} was typeset as 
  $+ 48h^2 \int_{-\infty}^{\infty} (L_t^x)^2 
      -  (\Delta_x^h L_t^x) L_t^x \diff{x}$ in the original
      reference~\cite{rosen_clt_2011}. In the statement of
      Conjecture~\ref{conj:4}, the author corrected this in a way that
      seemed the most reasonable at the time of writing.     } 

\begin{conjecture}[Rosen,~\cite{rosen_clt_2011}]
  \label{conj:4}
  Writing $\Delta_x^h L_t^x = L_t^{x+h}-L_t^x$, it holds that
  \begin{multline}
    \label{eq:47}
    \frac{1}{h^{5/2}}
    \left(
      \int_{-\infty}^{\infty} \left( \Delta_x^hL_t^x \right)^4 \diff{x}
      - 24h \int_{-\infty}^{\infty} \left( \Delta_x^h L_t^x \right)^2 L_t^x
      \diff{x}
      \right. \\ \left.
        + 48h^2 \int_{-\infty}^{\infty} (L_t^x)^2  \diff{x}
        - \int_{-\infty}^{\infty}  (\Delta_x^h L_t^x) L_t^x \diff{x}
    \right)
    \xrightarrow{d} \sqrt{\frac{2^{9} 4!}{5}  \int_{-\infty}^{\infty}
      (L_t^x)^4 \diff{x}} \, Z, 
  \end{multline}
  where $Z \sim \mathcal{N}(0,1)$, independent of $(L_t^x)_{x \in \R}$.
\end{conjecture}

Though the techniques are clearly different, all aforementioned proofs approach
the problem through the time domain (the variable $t$). For example, in Rosen's
method of moments proofs, intersection local times of the type
$\int_{-\infty}^{\infty} L_s^x \widetilde{L}_t^x \diff{x}$ are considered, where
$\widetilde{L}_t^x$ is the local time of another  
Brownian motion, independent of the one driving $L_t^x$, and in the Malliavin
calculus approach of Hu and Nualart the quantity $\int_{-\infty}^{\infty}
\Delta_x^h L_t^x \diff{x}$ is expressed as a stochastic integral indexed by $t$. 
In this paper, we take a fundamentally different approach and exclusively work
in space (the variable $x$), which seems to be more natural for the problem
at hand. This allows us to prove the following limit theorem for arbitrary
integer powers $q$.  

\begin{theorem}
  \label{thm:4}
  For integers $q \geq 2$ it holds that
  \begin{equation}
    \label{eq:11}
    \frac{1}{h^{\frac{q+1}{2}}}
    \left(
      \int_{-\infty}^{\infty} \left( L_t^{x+h} - L_t^x \right)^q 
      \diff{x} 
      +
      R_{q,h} \right)
    \xrightarrow{d}
    c_q
    \sqrt{
       \int_{-\infty}^{\infty} (L_t^x)^{q} \diff{x}
    } \, Z,
  \end{equation}
  where $Z$ is a standard Gaussian random variable, independent of
  $(L_t^x)_{x \in \R}$, the random variable $R_{q,h}$ is given by
  \begin{equation*}
    R_{q,h} = 
          \sum_{k=1}^{\lfloor \frac{q}{2} \rfloor} a_{q,k} 
      \int_{-\infty}^{\infty} \left( L_t^{x+h} - L_t^x \right)^{q-2k} \left(
        4 \int_x^{x+h} L_t^u \diff{u} \right)^k \diff{x}
  \end{equation*}
  and the constants $a_{q,k}$ and $c_q$ are defined as
  \begin{equation}
    \label{eq:8}
    a_{q,k} = \frac{(-1)^k q!}{2^k k! (q-2k)!} \qquad \text{and} \qquad
    c_q = \sqrt{\frac{2^{2q+1} q!}{q+1}}.
  \end{equation}
\end{theorem}

It will be shown later that $R_{2,h}=4ht$ and
$R_{3,h}=0$, so that for $r=0$, the known central limit theorems~\eqref{eq:9}
and~\eqref{eq:32}, corresponding to $q=2$ and $q=3$, respectively, are included
as special cases. Moreover, for $q=4$, Theorem~\ref{thm:4} proves Rosen's
Conjecture~\ref{conj:4} (with a slightly different compensator; see
Remark~\ref{rem:4} for details). Starting from $q=4$, however, the compensator
term $R_{q,h}$ becomes random and~\eqref{eq:11} thus no longer states a
\emph{central} 
limit theorem. To remedy the situation, one would have to prove that $R_{q,h}$
can be replaced by its expectation. Unfortunately, we have to leave this
problem for further research and again refer to Remark~\ref{rem:4} for a more
detailed discussion of this point.

Our approach allows for generalizations in several directions. For
example, as we never touch the time variable $t$ in our proofs, it can be
replaced with a suitable stopping time $\tau$ (see Theorem~\ref{thm:3}). 

The proof of Theorem~\ref{thm:4} can roughly be sketched as follows. The
starting point is the 
semimartingale decomposition of Browian local time in space, first proven by
Perkins in his celebrated paper~\cite{perkins_local_1982} and subsequently
refined by Jeulin in~\cite{jeulin_application_1985}). 
Decomposing  $L_t^{x}$ into its martingale and finite
variaton part $M_x$ and $V_x$, respectively (where we have
suppressed the dependence on $t$ for better legibility), some careful stochastic
analysis yields that  
\begin{equation*}
  \int_{-\infty}^{\infty} \left( L_t^{x+h} - L_t^x \right)^{q} \diff{x}
  \approx
  \int_{-\infty}^{\infty} \left( M_{x+h} - M_x \right)^{q} \diff{x}.
\end{equation*}
From here, through an iterative application of the Kailath-Segall formula 
(see~\cite{segall_orthogonal_1976}), the leading term of the integral on the 
right hand side turns out to be a certain iterated integral, whose limit can be
obtained by Pitman and Yor's asymptotic Ray-Knight Theorem
(see~\cite{pitman_asymptotic_1986}). This asymptotic Ray-Knight Theorem was also
used by Hu and Nualart 
in~\cite{hu_stochastic_2009} and~\cite{hu_central_2010} for their Malliavin
calculus proofs, with the notable difference that in their case the
Dambis-Dubins-Schwarz Brownian motion comes from a time change, while we obtain
it through a space change. 

The rest of this paper is organized as follows. In Section~\ref{s:prelim}, we
introduce some results from the literature and fix the notation which
is used throughout this paper. In Section~\ref{s:integrals}, several 
crucial types of iterated integrals are introduced, including a rather technical
analysis of their asymptotics. The proofs of Theorem~\ref{thm:4} and several 
generalizations are provided in Section~\ref{s:main}.

\section{Preliminaries and notation}
\label{s:prelim}

In the sequel we will deal with stochastic integrals of the form
$\int_{-\infty}^x Y_u \diff{M_u}$, where $(M_u)_{u \geq -\infty}$ is a
continuous martingale.
In general, constructing such stochastic integrals is a very
delicate task, which, to the best knowledge of the author, has only recently
been treated rigorously and in full generality by Basse-O'Connor, Graversen and
Pedersen  
in~\cite{basse-oconnor_martingale-type_2010}
and~\cite{basse-oconnor_stochastic_2014}. For example, an application of
Doob's backward martingale convergence theorem (see for example~\cite[p. 328,
Thm. 4.2]{doob_stochastic_1990}) shows that two sided Browian 
motion can neither be a martingale, nor a local martingale in any filtration and
thus can not be integrated against in the classical way. To overcome this
problem, one can introduce the notion of increment martingale
(see~\cite{basse-oconnor_martingale-type_2010}). 

However, in this paper all processes of the form $(X_u)_{u \geq -\infty}$
vanish at $-\infty$, so that stochastic integration can be defined in the
classical way, starting from simple processes and building up to
semimartingales, as for example outlined in the
monographs~\cite{revuz_continuous_1999,karatzas_brownian_1991,jacod_calcul_1979,
  rogers_diffusions_1994,rogers_diffusions_2000}. Due to the finite limit at
$-\infty$, the usual index interval $[0,\infty)$ can be replaced by
$\Set{-\infty} \cup \R$ and all tools from standard textbook martingale theory,
such as It\^o's formula, the Burkholder-Davis-Gundy inequality etc. remain
valid.

Let us define several classical spaces of martingales and 
integrators. We set
$H^2_{0}$ to be the space of $L^2$-bounded continuous martingales indexed by
$[-\infty,\infty)$ and vanishing at $-\infty$. Given $M \in
H^2_0$, the Hilbert space $L^2(M)$ consists of all (equivalence
classes of) progressively measurable processes $(Y_x)_{x \geq -\infty}$ such
that 
\begin{equation*}
  \norm{Y}_M = \Ex{\int_{-\infty}^{\infty} Y_u^2 \diff{\left\langle M
      \right\rangle_u}} < \infty,
\end{equation*}
where as usual $\left\langle M \right\rangle_u$ denotes the quadratic variation 
of $M$. A continuous local martingale $M$ belongs to $H^2_{0,\text{loc}}$, if
its localized version is an element of $H^2_0$ and, given a continuous local
martingale $M=(M_x)_{x\geq -\infty}$, the space $L^2_{\text{loc}}(M)$ contains
all progressively measurable processes $Y$ such that
\begin{equation*}
  \Ex{\int_{-\infty}^{T_n} Y_u^2 \diff{\left\langle M \right\rangle}_u} <
  \infty 
\end{equation*}
for a sequence $(T_n)$ of stopping times increasing to infinity. We will make
frequent use of the following result from stochastic analysis.

\begin{theorem}[Stochastic Fubini Theorem]
    Let $(A,\mathcal{A})$ be a measurable space, $\mu$ be a $\sigma$-finite,
  measure on $A$ and denote by $\mathcal{P}$ the predictable
  $\sigma$-algebra on $[-\infty,\infty) \times \Omega$. Furthermore, let $(M)_{x
    \geq -\infty} \in H_{0,\text{loc}}^2$ and $(X_{a,x})_{a \in A,x \in
    [-\infty,\infty)}$ be a continuous, $\mathcal{A} \otimes \mathcal{P}$-measurable 
  stochastic process indexed by $A \times [-\infty,\infty)$ which is
  $\mu$-integrable and assume that
  \begin{equation}
    \label{eq:16}
    \E \left[ 
    \int_{-\infty}^{\infty}
     \int_A^{} X_{a,u}^2 \mu(\diff{a}) \diff{\left\langle M,M
       \right\rangle}_u^{{T_n}}
      \right] < \infty,
  \end{equation}
  where $(T_n)_{n \geq 1}$ is a sequence of localizing stopping times for
  $M$. Then, for all $x \in [-\infty,\infty)$,
  \begin{equation}
    \label{eq:72}
    \int_{-\infty}^{x} \int_A^{} X_{a,u} \mu(\diff{a}) \diff{M_u}
    \overset{a.s.}{=}
    \int_A \int_{-\infty}^x X_{a,u} \diff{M_u} \, \mu(\diff{a}).
  \end{equation}
  In particular, the double integral on the right hand side exists.
\end{theorem}

\begin{proof}
  A proof for semimartingales indexed by $[0,\infty]$
  can for example be found in~\cite[Ch. IV,
  Thm. 64]{protter_stochastic_2004}. The adaptation to our setting is
  straightforward. 
\end{proof}

As already indicated, we also need (a slightly generalized version of) the
asymptotic Ray-Knight Theorem by Pitman and Yor
(see~\cite{pitman_asymptotic_1986} or the textbook~\cite[Ch. XIII,
Thm. 2.3]{revuz_continuous_1999}). Again, adapting  the original proof is
straightforward.  

\begin{theorem}[Asymptotic Ray-Knight Theorem,~\cite{pitman_asymptotic_1986}]
  \label{thm:ark}
   For $k \geq 2$ and $n \geq 1$ define a sequence $(M_{1,x}^n,M_{2,x}^n,
  \dots, M_{k,x}^n)$ of $k$-tuples of continuous local martingales
  $(M_{j,x}^n)_{x \geq -\infty} \in H^2_{0,\text{loc}}$ such that for fixed $j =
  1,\dots,k$ the limit $\left\langle M_j^n, M_j^n \right\rangle_{\infty}$
  is either infinite for all $n$, or finite for all $n$.
  After possibly enlarging the underlying probability space, each $M_j^n$
  posesses a Dambis-Dubins-Schwarz Brownian motion 
  $\beta^n_j$ and an associated time change $T_j^n(y)$, such that
  \begin{equation}
    \label{eq:10}
    M_{j,T_j^n(y)}^n = \beta^n_{j,y}
  \end{equation}
  for $y \geq 0$ and $1 \leq j \leq k$ (for a proof, see for
  example~\cite[Ch. V, Thm. 1.7]{revuz_continuous_1999}).
  If for $a \geq 0$ and $1 \leq i,j \leq k$ with $i \neq j$
  \begin{equation}
    \label{eq:7}
    \sup_{x \in [-\infty,a]} \abs{\qvar{M_i^n}{M_j^n}_x} \to 0, \qquad (n \to
    \infty) 
  \end{equation}
  in probability, then the $k$-dimensional process $\beta^n_y =
  (\beta_{1,y}^n,\beta_{2,y}^n, \dots,\beta_{k,y}^n)_{y \geq 0}$ converges in
  distribution to a $k$-dimensional Brownian motion $(\beta_y)_{y \geq 0}$.
\end{theorem}

Let us now collect some of the results from Perkins' celebrated
paper~\cite{perkins_local_1982}.  

\begin{theorem}[\cite{perkins_local_1982}]
  \label{lem:8}
  For any $t > 0$ there exists a right continuous filtration
  $(\mathcal{G}_x)_{x \in \R}$ such that $x \mapsto L_t^x$ is a continuous
  $\mathcal{G}_x$-semimartingale on $\R$ with quadratic variation $4
  \int_{-\infty}^{x} L_t^u \diff{u}$. 
  Furthermore, if $L_t^x = M_{t,x} + \int_{-\infty}^{x} A_{t,u} \diff{u}$ is its 
  canonical decomposition, then it holds for any $p \geq 1$ that
  \begin{equation*}
    \norm{L_t^{\ast}}_p < \infty, \qquad \int_{-\infty}^{\infty} \abs{A_{t,x}}
    \diff{x} < \infty \qquad \text{and} \qquad
  \int_{-\infty}^{\infty} \norm{A_{t,x}}_p \diff{x} < \infty,
\end{equation*}
where $L_t^{\ast} = \sup_{x \in \R} \abs{L_t^x}$.
\end{theorem}

As a corollary of Theorem~\ref{lem:8} we get the following explicit
semimartingale decomposition (see~\cite[Ch. II]{jeulin_application_1985} for
several extensions). 

\begin{corollary} \label{cor:2} 
  In the setting of Theorem~\ref{lem:8}, there exists a
  $\mathcal{G}_x$-Browian motion $(\beta_x)_{x \in \R}$, such that 
  \begin{equation}
    \label{eq:48}
    L_t^x =
    \begin{cases}
      L_t^0 + 2 \int_0^{x} \sqrt{L_t^u} \diff{\beta_u} + \int_0^x A_{t,u}
      \diff{u} & \text{if $x \geq 0$} \\
      L_t^0 - 2 \int_x^{0} \sqrt{L_t^u} \diff{\beta_u} - \int_x^0 A_{t,u}
      \diff{u} & \text{if $x < 0$}.
    \end{cases}
  \end{equation}
\end{corollary}

In fact, as the next Lemma shows, the integrability property of the $L^p$-norms
of the finite variation kernel $A_{t,x}$ are also true for the local time itself.

\begin{lemma}
  \label{lem:5}
  For $p\geq 1$ and $a > 0$, it holds that
  \begin{equation}
    \label{eq:75}
    \int_{-\infty}^{\infty} \norm{L_t^x}_p^a \diff{x} < \infty.
  \end{equation}
\end{lemma}

\begin{proof}
  For $t=0$ the assertion is trivial. If $t>0$, we have by scaling that
  $\norm{L_t^x}_p = \sqrt{t} \norm{L_1^{x/\sqrt{t}}}$. Thus, through a change
  of variables in the integral~\eqref{eq:75}, we can and do assume without loss
  of generality that $t=1$. Furthermore, if $B$ denotes the 
  underlying Brownian motion, we see that $\norm{L_t^x}_p = \norm{L_t^{-x}}_p$,
  as $-B$ has the same law as $B$. Thus, 
  \begin{equation*}
    \int_{-\infty}^{\infty} \norm{L_1^x}_p^a \diff{x}
    =
    2
    \int_0^{\infty} \norm{L_1^x}_p^a \diff{x}.
  \end{equation*}
  In~\cite{takacs_brownian_1995}, it is proved that 
  \begin{equation*}
    \norm{L_1^x}_p = \frac{2}{\sqrt{2 \pi}} \left(
      x^{p+1}
      \int_1^{\infty} \me^{-x^2y^2/2} (y-1)^p \diff{y}
    \right)^{1/p}
  \end{equation*}
  for $x >0$, which implies that there exist positive constants $\alpha$, $\beta$
  and $x_0 > 1$, such that 
  \begin{equation*}
    \norm{L_t^{x}}_p^a \leq \alpha \me^{-\beta x^2}
  \end{equation*}
  for all $x > x_0$.
  Consequently,
  \begin{equation*}
    \int_0^{\infty} \norm{L_1^x}_p^a \diff{x}
  = \int_{0}^{x_0} \norm{L_1^x}_p^a \diff{x} + \int_{x_0}^{\infty}
  \norm{L_1^x}_p^a 
  \diff{x} 
  <\infty.
  \end{equation*}
\end{proof}

\section{The iterated integrals $I_q$, $J_q$ and $K_q$}
\label{s:integrals}

In this section, we define and study three types of iterated integrals, two of
them, namely~$I_q$ and~$J_q$, stochastic, the third one deterministic with a
random kernel. As already indicated in the introduction, these integrals will
appear later through the Kailath-Segall formula.

\begin{definition}
  Let  $(M_u)_{u \geq -\infty} \in H^2_{0,\text{loc}}$ and $(Y(u))_{u \geq -\infty}
\in L^2_{\text{loc}}(M)$.
For $q \geq 0$, and $-\infty \leq x_1 < x_2$, the iterated 
integrals $I_q$, $J_q$ and $K_q$ are defined as follows.
\begin{align}
  \label{inti}
  I_0(x_1,x_2) &= 1, &I_{q+1}(x_1,x_2) &=  \int_{x_1}^{x_2} I_q(u) \diff{M_u},
  \\
  \label{intj}
  J_0(Y,x_1,x_2) &= Y(x_2),
&
  J_{q+1}(Y,x_1,x_2) &=  \int_{x_1}^{x_2} J_{q}(Y,x_1,x_2) \diff{M_u},
  \\
  \label{intk}
    K_0(Y,x_1,x_2) &= Y(x_2), &
  K_{q+1}(Y,x_1,x_2) &= \int_{x_1}^{x_2} K_q(Y,x_1,u) \diff{\left\langle M,M
                       \right\rangle_u}. 
\end{align}
\end{definition}
If $x_1=-\infty$, we drop it from the arguments, so that
$I_q(-\infty,x_2)$ becomes $I_q(x_2)$, $J_q(Y,-\infty,x_2)$ becomes $J_q(Yd
,x_2)$
and $K_q(Y,-\infty,x_2)$ becomes $K_q(Y,x_2)$.

As $I_q(x_1,x_2) = J_q(1,x_1,x_2)$, the definition of $I_q$ is
redundant. We have included it to improve the readability of 
subsequent results. Observe that for any integer $q \geq 1$, the integrals $I_q$
and 
$J_q$ are elements of $H^2_{0,\text{loc}}$. The following is a consequence of
the main findings in~\cite{segall_orthogonal_1976}.  

\begin{proposition}
  \label{prop:kailath-segall}
  For $x_1 \geq -\infty$, define the martingale $(M_{x_2})_{x_2 \geq x_1}$ by
  $M_{x_2} = I_1(x_1,x_2)$. Then, for  $q \geq 2$, the
  \emph{Kailath-Segall   identity}
  \begin{equation}
    \label{eq:1}
    q I_q(x_1,x_2) = I_{q-1}(x_1,x_2) M_{x_2} - I_{q-2}(x_1,x_2) \left\langle
      M,M \right\rangle_{x_2}
  \end{equation}
  holds. Furthermore, we have for $q \geq 1$ that
  \begin{equation}
    \label{eq:2}
    q! I_q(x_{1},x_2) =
    \widetilde{H}_q(M_{x_2},\left\langle M,M \right\rangle_{x_2})
    =
    \sum_{k=0}^{\lfloor \frac{q}{2} \rfloor} a_{q,k}
    M^{q-2k}_{x_2} 
    \left\langle M,M \right\rangle^k_{x_2}.
  \end{equation}
  Here, $\widetilde{H}_{n}(x,a) = a^{n/2} H_n(x/\sqrt{a})$, where $H_n(x)=
  \me^{\frac{x^2}{2}} \frac{\partial^n}{\partial x^n} \me^{-\frac{x^2}{2}}$
  denotes the $n$th Hermite polynomial and the constants $a_{q,k}$ are the
  Bessel of numbers of the second kind, given by
  \begin{equation}
    \label{eq:3}
    a_{q,k} = \frac{(-1)^k q!}{2^k k! (q-2k)!}.
  \end{equation}
\end{proposition}

\begin{proof}
   The identity $H_n(x)'=nH_{n-1}(x)$ and the 
  recursion formula
  \begin{equation*}
    H_{n}(x) = x H_{n-1}(x) - H_{n-1}(x)' = xH_{n-1}(x) -
  (n-1) H_{n-2}(x),
  \end{equation*}
  which are both well-known (see for
  example~\cite[p. 106]{szegho_orthogonal_1975}) and follow inductively from the
  definition, imply that
  \begin{equation}
    \label{eq:4}
    \widetilde{H}_n(x,a) = x \widetilde{H}_{n-1}(x,a) - (n-1) a
    \widetilde{H}_{n-2}(x,a). 
  \end{equation}
  and also that
  \begin{equation*}
    \left( \frac{1}{2}
    \frac{\partial^2}{\partial x^2} + \frac{\partial}{\partial t } \right)
  \widetilde{H}_n(x,a) = 0 \qquad \text{and} \qquad \frac{\partial}{\partial x}
  \widetilde{H}_n = n 
  \widetilde{H}_{n-1}.
  \end{equation*}
Thus (see for example~\cite[ch. IV, prop. 3.8]{revuz_continuous_1999}),
$\widetilde{H}_n(M_x,\left\langle M,M \right\rangle_x)$ is a martingale and
satisfies the recursion
  \begin{equation*}
    \widetilde{H}_n(M_x,\left\langle M,M \right\rangle_x) = n \int_{-\infty}^x
    \widetilde{H}_{n-1}(M_u,\left\langle M,M \right\rangle_u) \diff{M_u},
  \end{equation*}
  which inductively implies that
  \begin{equation}
    \label{eq:5}
    q! I_q(x_1,x_2) = \widetilde{H}_q(M_{x_2}, \left\langle M,M \right\rangle_{x_2}). 
  \end{equation}
  Together with~\eqref{eq:4}, the Kailath-Segall
  formula~\eqref{eq:1} follows. Identity~\eqref{eq:2} is a consequence
  of~\eqref{eq:5} as well, together with the well-known formula
  (see for example~\cite[p. 106]{szegho_orthogonal_1975})
  \begin{equation*}
    H_q(x) = \sum_{k=0}^{\lfloor \frac{q}{2} \rfloor} a_{q,k} x^{q-2k},
  \end{equation*}
  which translates into
  \begin{equation*}
    \widetilde{H}_q(x,a) = \sum_{k=0}^{\lfloor \frac{q}{2} \rfloor} a_{q,k} a^k
    x^{q-2k}. 
  \end{equation*}
\end{proof}

\begin{remark}\hfill
  \begin{enumerate}[(i)]
  \item     As indicated in the proof, the Bessel numbers $a_{q,k}$ are the
    coefficients 
    of the Hermite polynomials. The first few values of $a_{q,k}$ for $q,k \geq
    1$ (note that $a_{q,0}=1$) are given in the table below.
    \begin{center}
      \begin{tabular}{lrrrr}
        \toprule
        $q \textbackslash k$ &$1$ &$2$ &$3$ &$4$ \\
        \midrule
        $1$  &&&& \\
        $2$  &$1$ &&& \\
        $3$  &$3$ &&& \\
        $4$  &$6$ &$3$ &&\\
        $5$  &$10$ &$15$ &&\\
        $6$  &$15$ &$45$ &$15$ &\\
        $7$  &$21$ &$105$ &$105$ &\\
        $8$  &$28$ &$210$ &$420$ &$105$ \\
        $9$  &$36$ &$378$ &$1260$ &$945$ \\
        \bottomrule
      \end{tabular}
    \end{center}
    The third row for example translates into
    \begin{equation*}
      3! I_3 = I_1^3 - 3 I_1 \left\langle M \right\rangle
    \end{equation*}
    and row eigth shows that
    \begin{equation*}
      8! I_8 = I_1^8 - 28 I_1^6 \left\langle M \right\rangle + 210 I_1^4
      \left\langle M \right\rangle^2 - 420 I_1^2 \left\langle M
      \right\rangle^3 + 105 \left\langle M \right\rangle^4.
    \end{equation*}
  \item The Kailath-Segall identity~\eqref{eq:1} continues to hold verbatim for
    continuous semimartingales. For general semimartingales it involves
    higher-order variations (see~\cite{segall_orthogonal_1976}).
  \end{enumerate}
\end{remark}

The following lemma gives an analogue of the Burkholder-Davis-Gundy inequality
for the iterated integrals.

\begin{lemma}[\cite{carlen_$l^p$_1991}]
  \label{lem:2}
  Let $(M_x)_{x \geq -\infty} \in
  H^2_{0,\text{loc}}$. Then, for the iterated integrals $I_q$ defined with
  respect to $M$, we have that
  \begin{equation*}
    A_{p,q} \norm{\left\langle M,M \right\rangle_x^{1/2}}_{pq}^q
    \leq
    \norm{I_q(x)}_p
    \leq
    B_{p,q} \norm{\left\langle M,M \right\rangle_x^{1/2}}_{pq}^q,
  \end{equation*}
  where the left hand side holds for $p>1$, the right hand side for $p \geq 1$
  and $A_{p,q}$, $B_{p,q}$ denote positive constants depending on $p$ and $q$.
\end{lemma}

The proof, which was originially given for martingales indexed by $[0,\infty)$
and continues to work in our framework, uses the Kailath-Segall identity. The 
constants $A_{p,q}$ and $B_{p,q}$ can be computed explicitly, decay in $q$ and
are sharp in a certain sense (none of these facts are needed here,
see~\cite{carlen_$l^p$_1991} for details). Unfortunately, the approach can not
be adapted to cover the iterated integrals $J_q$ or $K_q$, but nevertheless, as
the next Lemma shows, we can  derive rather tight upper bounds which suffice for
our purposes. 

\begin{lemma}
  \label{lem:jq}
  Let $J_q$ and $K_q$ be the iterated integrals as defined above with respect to
  the local martingale part $M_x$ of the local time $L_t^x$. Then, for two real
  numbers $x_1 < x_2$, an integrable continuous process $(X_u)_{u \in
    [x_1,x_2]}$, $p \geq 1$ and any integer $q\geq0$, it holds
  that
  \begin{align}
    \label{eq:66}
    \norm{J_q(X,x_1,x_2)}_p
    &\leq
      C_p (x_2-x_1)^{\frac{q}{2}}
    \norm{L_t^{\ast}}_{2^qp}^{q/2}
     \sup_{u \in (x_1,x_2)}
     \norm{X_{u}}_{2^{q}p}
     \\
    \shortintertext{and}
    \label{eq:67}
    \norm{K_q(X,x_1,x_2)}_p
    &\leq
    (x_2-x_1)^q \norm{L_t^{\ast}}_{2^qp}^ q\sup_{u \in (x_1,x_2)}
      \norm{X_u}_{2^qp},
    \end{align}
  where $C_p$ denotes a positive constant and $L_t^{\ast} =
  \sup_{x \in \mathbb{R}} \abs{L_t^x}$. 
\end{lemma}

\begin{proof}
  We assume throughout the proof that $p \geq 2$. The general case follows a
  posteriori as the~$L^p$-norms with respect to a finite measure are increasing
  in $p$. Furthermore, the constant $C_p$ might change 
  from line to line. For $q=0$, both inequalities are trivially satisfied.
  Inductively, for $q \geq 1$, Burkholder-Davis-Gundy yields that
  \begin{multline*}
    \norm{J_q(X,x_1,x_2)}_p
     =
      \norm{
         \int_{x_1}^{x_2} J_{q-1}(X,x_1,u)
         \diff{M_u}
      }_p
       \\
    \leq
         C_p
         \norm{
    \left(\int_{x_1}^{x_2} J_{q-1}(X,x_1,u)^2
         L_t^u \diff{u} \right)^{1/2}
         }_p.
  \end{multline*}
  By Jensen's inequality, the above is less than or equal to
  \begin{equation*}
         C_p
         (x_2-x_1)^{1/2-1/p}
         \left(
         \int_{x_1}^{x_2}
         \E \left[
         \abs{J_{q-1}(X,x_1,u)}^p
         (L_t^u)^{p/2} \right] \diff{u}
         \right)^{1/p}
  \end{equation*}
  and by Cauchy-Schwarz and the induction hypothesis, this can further be
  bounded by 
  \begin{multline*}
         C_p
         (x_2-x_1)^{1/2-1/p}
         \left(
         \int_{x_1}^{x_2}
         \norm{J_{q-1}(X,x_1,u)}_{2p}^{p}
         \norm{L_t^u}_p^{p/2}
         \diff{u}
         \right)^{1/p}
        \\ \leq
         C_p
         (x_2-x_1)^{q-2}
         \norm{L_t^{\ast}}_p^{1/2}
         \norm{L_t^{\ast}}_{2^qp}^{(q-1)/2}
         \sup_{u \in (x_1,x_2)} \norm{X_u}_{2^qp}.
  \end{multline*}
  Together with the monotonicity of the
  $L^p$ norms with respect to a finite measure, this proves
  inequality~\eqref{eq:66}. 
  To show~\eqref{eq:67}, apply Jensen's inequality to
  obtain 
  \begin{multline*}
    \norm{K_q(X,x_1,x_2)}_p
    =
         \norm{\int_{x_1}^{x_2} K_{q-1}(X,x_1,u) L_t^u \diff{u}}_p
    \\ \leq
         (x_2-x_1)^{1-\frac{1}{p}}
         \left(
         \int_{x_1}^{x_2} \E \left[
         \abs{K_{q-1}(X,x_1,u) L_t^u}^p 
       \right]
              \diff{u}
       \right)^{1/p}
  \end{multline*}
  and then again use Cauchy-Schwarz and the induction hypothesis.
\end{proof}

We now turn to the proof of two crucial technical Lemmas.

\begin{lemma}
  \label{lem:3}
  For an integer $q \geq 2$, let $J_{q-1}$ be the iterated integral defined
  with respect to the local martingale  part $M_x$ of the local time $L_t^x$. Furthermore,
  let $(Y_x)_{x     \in \R}$ be a uniformly bounded continuous stochastic
  process adapted to $(L_t^x)_{x \in 
    \R}$ and define $X_u(v) = \int_u^v Y_x \diff{x}$. Then, for any
   $x_0 \in  \R$ it holds   that
  \begin{equation}
    \label{eq:13}
    \E \left[ 
    \sup_{x \in (-\infty,x_0)}
    \frac{1}{h^{(q+1)/2}}
    \abs{
      \int_{-\infty}^x J_{q-1}(X_{u-h},u-h,u) L_t^u \diff{u}}
 \right]
\to 0.
  \end{equation}
\end{lemma}

\begin{proof}
    Throughout the proof, $C$ and $C_p$ denote positive constants, the latter
    depending on $p$, which might change from line to line.
  We have that
  \begin{align}
    \notag
    \MoveEqLeft \frac{1}{h^{(q+1)/2}}
      \int_{-\infty}^x J_q(X_{u-h},u-h,u) L_t^u \diff{u}
      \\ & \label{eq:36}
           \begin{multlined}[c][.85\columnwidth]
             =
             \frac{1}{h^{(q+1)/2}}
                \int_{-\infty}^x J_{q-1}(X_{u-h},u-h,u) L_t^{u-h} \diff{u}
              \\ + \frac{1}{h^{(q+1)/2}}
             \int_{-\infty}^x J_{q-1}(X_{u-h},u-h,u) \left(L_t^u -
               L_t^{u-h}
             \right) \diff{u}.
           \end{multlined}
  \end{align}
  In two separate steps, we show that the supremum of each of the two integrals
  in the sum~\eqref{eq:36} converges to zero in~$L^1$.
  \\
  \textit{Step 1.}
    An iterated application of the stochastic Fubini theorem, justified by
    Lemma~\ref{lem:jq}, yields that
    \begin{multline}
      \label{eq:68}
      \int_{-\infty}^{x} J_{q-1}(X_{u_1-h},u_1-h,u_1) L_t^{u_1-h}
                               \diff{u_1} 
     \\ =
           \int_{-\infty}^x J_{q-2}(\widetilde{X}_{u_2-h,x-h},u_2-h,u_2)
           \diff{M_{u_2}}, 
    \end{multline}
    where
    \begin{align*}
      \widetilde{X}_{u_2-h,x-h}(u_{q}) = \int_{u_2-h}^{u_{q} \land (x-h)}
      \int_{u_1}^{u_{q}} Y_{u_{q+1}} \diff{u_{q+1}} L_t^{u_1}
      \diff{u_1}.
    \end{align*}
    Note that for $u_q \in (u_2-h,u_2)$ it holds that
    \begin{equation*}
       \abs{\widetilde{X}_{u_2-h,x-h}(u_q)} \leq C h^2 
    \end{equation*}
    and thus, by Lemma~\ref{lem:jq},
    \begin{equation}
      \label{eq:69}
      \norm{ J_{q}(\widetilde{X}_{u_2-h,x-h},u_2-h,u_2)}_p
      \leq
      C_p h^{q/2 + 2} \norm{L_t^{\ast}}_{2^qp}^{q/2}
    \end{equation}
    Abbreviating  $\widetilde{X}_{u_2-h,x-h}$ by $\widetilde{X}$, we use
    identity~\eqref{eq:68}, Burkholder-Davis-Gundy and the deterministic 
    Fubini theorem to obtain
    \begin{align}
      \notag
      \MoveEqLeft[6]
      \norm{
      \sup_{x \leq x_0}
      \abs{
      \int_{-\infty}^x J_{q-1}(X_{u_1-h},u_1-h,u_1) L_t^{u_1-h}
      \diff{u_1}       
        }}_2
      \\ &=           \norm{
           \sup_{x \leq x_0}
           \abs{
           \int_{-\infty}^x
           J_{q-2}(\widetilde{X},u_2-h,u_2) \diff{M_{u_2}}
           }
           }_2
      \\ &\leq C \notag
           \norm{
           \left(
           \int_{-\infty}^{\infty}
           J_{q-2}(\widetilde{X},u_2-h,u_2)^2 L_t^{u_2} \diff{u_2}
           \right)^{1/2}
           }_2
      \\ &\leq C \notag
           \left(
           \int_{-\infty}^{\infty}
           \norm{J_{q-2}(\widetilde{X},u_2-h,u_2)}_4^2
           \norm{L_t^{u_2}}_2 \diff{u_2}
           \right)^{1/2}
      \\
      \shortintertext{and using~\eqref{eq:69}, we can continue to write}
      &\leq
        C h^{(q+2)/2}
        \norm{L_t^{\ast}}_{2^{q}}^{q/2+r}
        \int_{-\infty}^{\infty} \norm{L_t^{u_2}}_2 \diff{u_2}.
    \end{align}
    \textit{Step 2}. Before treating the second integral of~\eqref{eq:36}, 
    note that
    \begin{equation}
      \label{eq:61}
      \abs{X_{u}(v)} = \abs{\int_u^{v} Y_x \diff{x}} \leq C \, (v-u)
    \end{equation}
    and thus, by Lemma~\ref{lem:jq},
    \begin{equation}
      \label{eq:17}
      \norm{J_q(X_{u-h},u-h,u)}_p \leq
      C_p
      h^{q/2 + 1} \norm{L_t^{\ast}}_{2^qp}^{q/2}.
    \end{equation}
    By Cauchy-Schwarz, it holds that
    \begin{align*}
      \MoveEqLeft[2] \abs{
      \int_{-\infty}^x J_{q-1}(X_{u-h},u-h,u) \left(L_t^u -
               L_t^{u-h}
      \right) \diff{u}}
      \\ &\leq
           \left(
           \int_{-\infty}^{\infty} J_{q-1}(X_{u-h},u-h,u)^2 \diff{u}
           \right)^{1/2}
           \left(
      \int_{-\infty}^{\infty} \abs{L_t^u - L_t^{u-h}}^{2}
           \diff{u} \right)^{1/2}
      \\ &=
           \left( h
           \int_{-\infty}^{\infty} J_{q-1}(X_{u-h},u-h,u)^2 \diff{u}
           \right)^{1/2}
           \left( \frac{1}{h}
      \int_{-\infty}^{\infty} \abs{L_t^u-L_t^{u-h}}^2
           \diff{u} \right)^{1/2}.
    \end{align*}
    Taking supremum and expectation and applying Cauchy-Schwarz once again
    yields 
    \begin{multline}
      \label{eq:59}
      \E \left[ \sup_{x \leq x_0} \abs{
           \int_{-\infty}^x J_{q-1}(X_{u-h},u-h,u) \left(L_t^u -
               L_t^{u-h}
             \right) \diff{u}} \right]
         \\ \leq
         h^{1/2}
         \E \left[
             \int_{-\infty}^{\infty} J_{q-1}(X_{u-h},u-h,u)^2 \diff{u}
           \right]^{1/2}
           \\
           \E \left[ \frac{1}{h}
                   \int_{-\infty}^{\infty}
                   \abs{L_t^u-L_t^{u-h}}^2
           \diff{u} 
           \right]^{1/2}.
         \end{multline}
      By~\cite[Thm. 1.1]{marcus_$l^p$_2008}, the second expectation on the right
      hand side converges to a bounded quantity, so we are done if we can
      show that the square root of the first is of order $h^\varepsilon$ for
      some $\varepsilon > q/2$.
    An application of It\^o's formula gives
    \begin{multline*}
      J_{q-1}(X_{u-h},u-h,u)^2
      \\ =
          \int_{u-h}^u J_{q-1}(X_{u-h},u-h,v) J_{q-2}(X_{u-h},u-h,v) \diff{M_v}
      \\+
      \int_{u-h}^u J_{q-2}(X_{u-h},u-h,v)^2 \diff{\left\langle M,M
        \right\rangle_v}.
     \end{multline*}
    Stochastic Fubini and Burkholder-Davis-Gundy yield that
    \begin{align*}
      \MoveEqLeft[2] 
      \E \left[  \int_{-\infty}^{\infty}  \int_{u-h}^u J_{q-1}(X,u-h,v)
        J_{q-2}(X,u-h,v) \diff{M_v} \diff{u}  \right]
      \\ &= \notag
           \E \left[
           \int_{-\infty}^{\infty}
           \int_{v}^{v+h}
            J_{q-1}(X,u-h,v)J_{q-2}(X,u-h,v)
           \diff{u}
            \diff{M_v}
            \right]
      \\ &\leq \notag
           \E \left[
           \int_{-\infty}^{\infty}
           \left(
           \int_{v}^{v+h}
           J_{q-1}(X,u-h,v) J_{q-2}(X,u-h,v) 
           \diff{u}
           \right)^2
           4 L_t^v
           \diff{v}
           \right]^{1/2},
    \end{align*}
    and several applications of deterministic Fubini and the Cauchy-Schwarz
    inequality further bound the preceding by
    \begin{multline*}
           2h^{3/4}
           \bigg(
           \int_{-\infty}^{\infty}
           \left(
           \int_{v}^{v+h}
           \norm{J_{q-1}(X,u-h,v)}_8^4 \norm{J_{q-2}(X,u-h,v)}_8^4
           \diff{u}
         \right)^{1/2}
         \\
           \norm{L_t^v}_2
           \diff{v}
           \bigg)^{1/2}.
         \end{multline*}
    Together with~\eqref{eq:69}, we arrive at
    \begin{equation}
      \label{eq:63}
      \begin{multlined}[.85\columnwidth]
              \E \left[  \int_{-\infty}^{\infty}  \int_{u-h}^u J_{q-1}(X,u-h,v)
                J_{q-2}(X,u-h,v) \diff{M_v} \diff{u}  \right]
              \\
      \leq
                 C h^{q+3/2}
           \norm{L_t^{\ast}}_{2^{q+2}}^{q-3/2}
           \left(
           \int_{-\infty}^{\infty}
           \norm{L_t^v}_2 \diff{v} \right)^{1/2}.
      \end{multlined}
    \end{equation}
    Similarly, using Fubini, Jensen's inequality and~\eqref{eq:69}, we can show
    that
    \begin{equation}
      \label{eq:64}
      \begin{multlined}[.85\columnwidth]
              \E \left[ \int_{-\infty}^{\infty}\int_{u-h}^u J_{q-2}(X,u-h,v)^2
        \diff{\left\langle M,M \right\rangle_v}
        \diff{u} \right]
     \\  \leq 
           C h^{q+1}
           \norm{L_t^{\ast}}_{2^{q-1}}^{q-2}
           \int_{-\infty}^{\infty}
           \norm{L_t^v}_2
           \diff{v}.
      \end{multlined}
    \end{equation}
    Plugging~\eqref{eq:63} and~\eqref{eq:64} back into~\eqref{eq:59} yields
    that
    \begin{equation*}
            \E \left[ \sup_{x \leq x_0} \abs{
           \int_{-\infty}^x J_{q-1}(X_{r,u-h},u-h,u) \left(L_t^u -
               L_t^{u-h}
             \right) \diff{u}} \right]^{1/2}
       \end{equation*}
    is of order $h^{(q+1)/2}$, concluding the proof.
\end{proof}

\begin{lemma}
  \label{lem:4}
  Let $J_{q-1}$ be defined with respect to the local martingale part $M_x$ of
  the local time~$L_t^x$ and $X_{u}(v) = \int_u^v Y_{x} \diff{x}$, where for
  some $\alpha>0$, $(Y_x)_{x \in \R}$ is a uniformly bounded, almost surely
  $\alpha$-H\"older continuous stochastic process adapted to $(L_t^x)_{x \in
    \R}$. Then, for any positive integer $q \geq 2$ and $x\in \R$, it holds that
  \begin{equation}
    \label{eq:14}
    \begin{multlined}[c][.85\columnwidth]
    \on{E} \left[ \left|
    \frac{1}{h^{q+1}} \int_{-\infty}^x J_{q-1}(X_{u-h},u-h,u)^2
    \diff{\left\langle M,M \right\rangle_u} \right. \right.
   \\-
  \left. \left.
      \frac{2^{2q+1}}{(q+1)!} \int_{-\infty}^x \left( L_t^u \right)^{q} Y_u^2
      \diff{u} 
    \right| \right]  \to 0.
    \end{multlined}
  \end{equation}
\end{lemma}

\begin{proof}
  Throughout the proof, the first argument of $J_{q-1}$ will not
  change. For better readability, we will drop it and for example write    
  $J_{q-1}(u-h,u)$ instead of $J_{q-1}(X_{u-h},u-h,u)$. As before, $C$ and $C_p$
  denote positive constants, the latter depending on $p$, which might 
  change from line to line.   
  By It\^o's formula,
  \begin{align*}
    J_{q-1}(u-h,u)^2  
    &=
    \left(
      \int_{u-h}^{u} J_{q-2}(u-h,v)
      \diff{M_{v}}
    \right)^2
    \\ &
         \begin{multlined}[.65\columnwidth]
           =
         2
         \int_{u-h}^{u} J_{q-1}(u-h,v) J_{q-2}(u-h,v)
         \diff{M_{v}}
         \\ +
         \int_{u-h}^{u} J_{q-2}(u-h,v)^2 \diff{\left\langle M,M
         \right\rangle_v}.
         \end{multlined}
  \end{align*}
  Recursively, this yields the identity
  \begin{equation*}
    \label{eq:34}
    J_{q-1}(u-h,u)^2 = 2 \sum_{j=0}^{q-2} K_j(U_{j,q-1,u-h},u-h,u)
    + K_{q-1}(X_{u-h}^2,u-h,u),
  \end{equation*}
  where
  \begin{equation*}
    U_{j,q-1,u-h}(v) = \int_{u-h}^{v} J_{q-1-j}(u-h,v) J_{q-2-j}(u-h,v)
    \diff{M_v}. 
  \end{equation*}
 
  Let us first show that, for $0 \leq j \leq q-2$,
  \begin{equation}
    \label{eq:38}
     \E \left[ \frac{1}{h^{q+1}} \abs{\int_{-\infty}^x
         K_{j}(U_{j,q-1,u-h},u-h,u) 
      \diff{\left\langle M,M \right\rangle_u}}  \right]
    \to 0 
  \end{equation}
  and then, in a second step, that
  \begin{equation}
    \label{eq:41}
    \begin{multlined}[.85\columnwidth]
    \E \left[ \left|
      \frac{1}{h^{q+1}} \int_{-\infty}^x K_{q-1}(X_{u-h}^2,u-h,u)
      \diff{\left\langle M,M \right\rangle_u} \right. \right.
    \\ \left. \left.
       -
      \frac{2^{2q+1}}{(q+1)!} \int_{-\infty}^x (L_t^u)^{q} Y_u^2 \diff{u}
    \right| \right] \to 0
    \end{multlined}
  \end{equation}
  as $h \to 0$.
  
  \textit{Step 1.}
  It holds that $\abs{X_{u}(v)} \leq C \, (v-u)$ and thus,
  by Lemma~\ref{lem:jq}, 
  \begin{equation}
    \label{eq:23}
    \norm{J_q(u-h,u)}_p \leq C_p h^{q/2 + 1} \norm{L_t^{\ast}}_{2^qp}^{q/2}.
  \end{equation}
  Together with Burkholder-Davis-Gundy and Jensen's inequality,
  this implies for $p \geq 2$ that 
  \begin{align}
    \notag
    \MoveEqLeft[1]
    \norm{U_{j,q-1,u_1-h}(u_1)}_p
    \\ &\leq  \notag
         C_p
      \norm{ \left(\int_{u_1-h}^{u_1} J_{q-1-j}(u_1-h,u_2)^2
      J_{q-2-j}(u_1-h,u_2)^2 L_t^{u_2} 
      \diff{u_2} \right)^{1/2} }_p
    \\ &\leq \notag
         C_p
         h^{\frac{1}{2}-\frac{1}{p}}
    \\ &\qquad  \notag
         \left(\int_{u_1-h}^{u_1}
                  \E \left[ 
         \abs{J_{q-1-j}(u_1-h,u_2)
         J_{q-2-j}(u_1-h,u_2)}^p
         (L_t^{u_2})^{p/2}
          \right]
         \diff{u_2} \right)^{1/p}
    \\ &\leq \notag
         C_p
         h^{\frac{1}{2}-\frac{1}{p}}
                  \norm{L_t^{\ast}}_p^{1/2}
    \\ &\qquad \notag 
         \left(\int_{u_1-h}^{u_1}
         \norm{J_{q-1-j}(u_1-h,u_2)}^{p}_{4p}
         \norm{J_{q-2-j}(u_1-h,u_2)}^{p}_{4p}
         \diff{u_2} \right)^{1/p}
   \\ &\label{eq:uj} \leq
         C_p
         h^{q -j + 1}
         \norm{L_t^{\ast}}_{2^{q+1-j}p}^{q-j-1}.
  \end{align}
  Using this result, Jensen's inequality and Lemma~\ref{lem:jq} gives
  \begin{equation}
    \label{eq:42}
    \begin{multlined}[.85\columnwidth]
   \E \left[
      \abs{
    \int_{-\infty}^x  K_j  (U_{j,q-1,u-h},u-h,u) \diff{\left\langle M,M
      \right\rangle_u} 
      } \right]
    \\ = 4 
         \int_{-\infty}^x
        \E \left[
         \abs{K_j  (U_{j,q-1,u-h},u-h,u) L_t^u} 
         \right]
         \diff{u}
    \\ \leq
         C_p
         h^{q + 1}
         \norm{L_t^{\ast}}_{2^{q+2}}^{q+1}
         \int_{-\infty}^x \norm{L_t^u}_2 \diff{u},
    \end{multlined}
  \end{equation}
  which shows that the expectation appearing in~\eqref{eq:38} is
  bounded. 
  Before continuing to show that it actually vanishes in the limit, let us 
  informally describe the technique we are going to use.
  It is an important observation that if instead of one of the local times 
  $L_t^{u_j}$ inside of
  $K_j$ or $U_{j,q-1,u_1-h}$ we would have encountered a difference of the form
  $L_t^{u_j} - L_t^{u_j -  a}$ in the above calculations, where $\abs{a} < h$,
  then the corresponding norm of $L_t^{\ast}$, appearing as a factor in the
  bound on the right hand side 
  of~\eqref{eq:42} would instead be the norm of the increment and could thus be
  bounded by $h^{\varepsilon}$ for any $\varepsilon \in (0,1/2)$, increasing the order 
  of the right hand side of~\eqref{eq:42} to $h^{q+1+\varepsilon}$. Mutatis
  mutandis, the same argument is valid for the H\"older continuous process
  $Y_u$.  
  By linearity of the (stochastic) integral, this reasoning  allows us to
  replace $L_t^{u_j}$ by $L_t^{u_j-a} +  (L_t^{u_j} - L_t^{u_j-a})$ and
  $Y_{u_j}$ by $Y_{u_j-a} + (Y_{u_j} - Y_{u_j-a})$ at any place in the
  calculations above, at the cost of introducing a negligible
  summand. In what follows, we will make frequent use of this fact, in order to
  nudge the processes occuring in the iterations of $K_j$ back in space and
  make them adapted to the  Brownian motion driving the stochastic integral
  inside $U_{j,q-1,u_1-h}$. This enables us to iteratively apply the stochastic
  Fubini theorem and bring this Brownian motion to the very outside, effectively
  replacing the stochastic differential by a deterministic one and increasing
  the order of convergence by a factor of $h^{1/2}$, which suffices to conclude
  that~\eqref{eq:14} holds. Restating these arguments in a more rigorous way, we
  claim that for $0 \leq j \leq q-1$, the $L^1$-norms of the integrals 
  \begin{multline}
    \label{eq:44}
    \int_{-\infty}^x K_j(U_{j,q,u-h},u-h,u) \diff{\left\langle M,M
      \right\rangle_u}
    \\ =
    4
   \int_{-\infty}^x K_j(U_{j,q,u-h},u-h,u) L_t^u \diff{u}
  \end{multline}
  are of order $o(h^{q+1-\varepsilon})$ for any $\varepsilon \in (0,1/2)$. To
  prove this claim for $j=0$, note that
  \begin{align}
    \MoveEqLeft[5]
    \notag
    \int_{-\infty}^x K_0(U_{0,q-1,u_1-h},u_1-h,u_1) L_t^{u_1} \diff{u_1}
    \\ &= \notag
      \int_{-\infty}^x U_{0,q-1,u_1-h}(u_1) L_t^{u_1} \diff{u_1}
    \\ & \label{eq:30}
         \begin{multlined}[.65\columnwidth]
         =\int_{-\infty}^x U_{0,q-1,u_1-h}(u_1) L_t^{u_1-h} \diff{u_1}
         \\ +
         \int_{-\infty}^x U_{0,q-1,u_1-h}(u_1) \left(L_t^{u_1} -L_t^{u_1-h}
         \right) 
         \diff{u_1}.
         \end{multlined}
  \end{align}
  For the first integral on the right hand side of~\eqref{eq:30}, stochastic
  Fubini yields  
  \begin{align*}
    \MoveEqLeft[3]
    \int_{-\infty}^x U_{0,q-1,u_1-h}(u_1) L_t^{u_1-h} \diff{u_1}
    \\ &=
         \int_{-\infty}^x \int_{u_1-h}^{u_1} J_{q-1}(u_1-h,u_2) J_{q-2}(u_1-h,u_2)
         \diff{M_{u_2}}
         L_t^{u_1-h} \diff{u_1}
    \\ &=
     \int_{-\infty}^x \int_{u_2}^{(u_2+h) \land x} J_{q-1}(u_1-h,u_2)
         J_{q-2}(u_1-h,u_2)
         L_t^{u_1-h} \diff{u_1}         
         \diff{M_{u_2}}.
  \end{align*}
  A straightfoward application of the Burkholder-Davis-Gundy, Cauchy-Schwarz and
  Jensen inequalities, as well as Lemma~\ref{lem:jq}, thus yields
  so that by Burkholder-Davis-Gundy and deterministic Fubini
  \begin{align*}
    \MoveEqLeft[1] \E \left[
      \abs{
      \int_{-\infty}^x U_{0,q-1,u_1-h}(u_1) L_t^{u_1-h} \diff{u_1}
      }
      \right] 
    \\ &=
         \E \left[
        \abs{
              \int_{-\infty}^x \int_{u_2}^{(u_2+h) \land x} J_q(u_1-h,u_2)
         J_{q-1}(u_1-h,u_2)
         L_t^{u_1-h} \diff{u_1}         
         \diff{M_{u_2}}
         }
         \right]
    \\ &\leq
                  C_p h^{2q+3/2}
         \norm{L_t^{\ast}}_{2^{q+1}}^{2q+3/2}
         \left(\int_{-\infty}^x \norm{L_t^{u_2}}_2 \diff{u_2}
         \right)^{1/2}.
  \end{align*}
  To treat the second integral on the right hand side of~\eqref{eq:30}, use
  Cauchy-Schwarz to get
  \begin{multline}
    \label{eq:35}
            \E \left[  \abs{
      \int_{-\infty}^x
    U_{0,q-1,u_1-h}(u_1)
    \left(L_t^{u_1} -L_t^{u_1-h} \right) 
    \diff{u_1}}
      \right]
    \\ \leq
         \E \left[ h
         \int_{-\infty}^x
         (U_{0,q-1,u_1-h}(u_1))^2
         \diff{u_1}
         \right]^{1/2} \\
         \E \left[ \frac{1}{h}
           \int_{-\infty}^x
           \left( L_t^{u_1} - L_t^{u_1-h} \right)^2
         \diff{u_1}
          \right]^{1/2}.
  \end{multline}
  By~\cite[Thm. 1.1]{marcus_$l^p$_2008}, the second expectation converges to a
  bounded quantity. It\^o's formula, applied to the integrand of the first,
  gives 
  \begin{align*}
    &U_{0,q-1,u_1-h}(u_1)^2
    \\ &=
    \left(
    \int_{u_1-h}^{u_1} J_{q-1}(u_1-h,u_2) J_{q-2}(u_1-h,u_2) \diff{M_{u_2}}
      \right)^2
    \\ &=
         \begin{multlined}[t][.85\columnwidth]
                2 \int_{u_1-h}^{u_1} U_{0,q-1,u_1-h}(u_2) J_{q-1}(u_1-h,u_2)
         J_{q-2}(u_1-h,u_2) \diff{M_{u_2}}
          \\+
         \int_{u_1-h}^{u_1} J_{q-1}(u_1-h,u_2)^2 J_{q-2}(u_1-h,u_2)^2
         \diff{\left\langle M,M \right\rangle_{u_2}}
       \end{multlined}
    \\ &= V_1 + V_2.
  \end{align*}
  By stochastic Fubini, we obtain that
  \begin{multline*}
    \E \left[ \int_{-\infty}^x V_1 \diff{u_1} \right]
    = 2
         \E \bigg[
         \int_{-\infty}^x
         \int_{u_2}^{u_2+h} U_{0,q-1,u_1-h}(u_2)
               \\ 
         J_{q-1}(u_1-h,u_2)
             J_{q-2}(u_1-h,u_2)
         \diff{u_1}
         \diff{M_{u_2}}
         \bigg].
   \end{multline*}
   Thus, after  straightforward application of Burkholder-Davis-Gundy and Jensen's
   inequality, the estimates~\eqref{eq:23} and~\eqref{eq:uj} yield the bound
   \begin{equation*}
     \E \left[ \int_{-\infty}^x V_1 \diff{u_1} \right]
     \leq
     C h^{2q+5/2} \norm{L_t^{\ast}}_{2^{q+5}}^{2(q+r-2)}
     \left(
       \int_{-\infty}^x \norm{L_t^{u_2}}_2 \diff{u_2}
     \right)^{1/2}.
   \end{equation*}
   The same arguments also work for $\Ex{\int_{-\infty}^x (-V_1) \diff{u_1}}$,
   so that we obtain
   \begin{equation*}
     \E \left[ \abs{\int_{-\infty}^x V_1 \diff{u_1}} \right]
     \leq
     C h^{2q+5/2} \norm{L_t^{\ast}}_{2^{q+5}}^{2(q+r-2)}
     \left(
       \int_{-\infty}^x \norm{L_t^{u_2}}_2 \diff{u_2}
     \right)^{1/2}.
   \end{equation*}
   Similarly,
   \begin{equation*}
     \E \left[ \abs{\int_{-\infty}^x  V_2 \diff{u_1}} \right]
     \leq
          C h^{2q+2} \norm{L_t^{\ast}}_{2^{q+2}}^{2q+4r-3}  \int_{-\infty}^x
          \norm{L_t^{u_2}}_2 \diff{u_2}. 
   \end{equation*}
   Plugged back into~\eqref{eq:35}, we see that
   \begin{equation*}
                 \E \left[  \abs{
      \int_{-\infty}^x
    U_{0,q-1,u_1-h}(u_1)
    \left(L_t^{u_1} -L_t^{u_1-h} \right) 
    \diff{u_1}}
      \right]
    \end{equation*}
    is of order $\mathcal{O}(h^{q+3/2})$,  concluding the proof for $j=0$.
   To obtain the asymptotic order of~\eqref{eq:44} for $j \geq 1$, we write
   \begin{equation}
     \label{eq:43}
     \begin{multlined}[.85\columnwidth]
    \int_{u_1 -h}^{u_j} U_{j,q-1,u_1-h}(u_{j+1})
    L_t^{u_{j+1}} \diff{u_{j+1}}
    \\ =
        \int_{u_1 -h}^{u_j} U_{j,q-1,u_1-h}(u_{j+1})
        L_t^{u_{l+1}-u_j+u_1-h} \diff{u_{j+1}}
        +
        R_h,
     \end{multlined}
  \end{equation}
  where
  \begin{equation*}
    R_h =  \int_{u_1 -h}^{u_j} U_{j,q-1,u_1-h}(u_{j+1})
        \left(L_t^{u_{j+1}}- L_t^{u_{l+1}-u_j+u_1-h} \right) \diff{u_{j+1}}.
  \end{equation*}
  Note that $R_h$, when plugged back into the
  integral~\eqref{eq:44}, introduces a negligible summand (see~\eqref{eq:42} and
  the arguments afterwards). 
  The stochastic process
  \begin{equation*}
    u_{j+2} \mapsto
    \left(
     \int_{ u_{j+2}}^{u_j} L_t^{u_{j+1}-u_j+u_1-h}
     \diff{u_{j+1}}
     \right), \qquad u_1-h \leq u_{j+2} \leq {u_j}   
 \end{equation*}
 is by construction adapted to
  $(\mathcal{F}_{u_{j+2}})_{u_1-h \leq u_{j+2} \leq {u_j}}$ as
  $u_{j+1}-u_j+u_1-h \leq u_1-h$, where $(\mathcal{F}_x)_{x \in \R}$ denotes the
  filtration of the underlying probability space. Therefore, we can apply 
  stochastic Fubini and get
  \begin{align*}
    \notag
    \int_{u_1 -h}^{u_j} &U_{j,q-1,u_1-h}(u_{j+1})
    L_t^{u_{j+1}-u_j+u_1-h} \diff{u_{j+1}}
    \\ &= \notag
         \begin{multlined}[t][.7\columnwidth]
         \int_{u_1-h}^{u_j}
         \int_{u_1-h}^{u_{j+1}}
         J_{q-1-j}(u_1-h,u_{j+2}) J_{q-2-j}(u_1-h,u_{j+2})
         \\
         \diff{M_{u_{j+2}}}
         L_t^{u_{j+1}-u_j+u_1-h} \diff{u_{j+1}}
       \end{multlined}
    \\ &= \notag
         \begin{multlined}[t][.7\columnwidth]
         \int_{u_1-h}^{u_j}
         \int_{u_{j+2}}^{u_j}
         L_t^{u_{j+1}-u_j+u_1-h} \diff{u_{j+1}}
         \\
         J_{q-1-j}(u_1-h,u_{j+2}) J_{q-2-j}(u_1-h,u_{j+2})
         \diff{M_{u_{j+2}}}.
       \end{multlined}
  \end{align*}
  Iterating this procedure of interchanging a deterministic and a stochastic
  integral at the cost of introducing negligible terms, we
  obtain that
  \begin{align*}
   \MoveEqLeft  K_j(U_{j,q-1,u_1-h},u_1-h,u_1) 
    \\ &=
      \begin{multlined}[t][.8\columnwidth]
        \int_{u_1-h}^{u_1} \dots \int_{u_1-h}^{u_j} U_{j,q-1,u_1-h}(u_{j+1})
        L_t^{u_{j+1}} \diff{u_{j+1}} \dots L_t^{u_1} \diff{u_2}
      \end{multlined}
    \\ &=
         \begin{multlined}[t][.8\columnwidth]
           \int_{u_1-h}^{u_1}
           \int_{u_{j+2}}^{u_1} \int_{u_{j+2}}^{u_2} \dots \int_{u_{j+2}}^{u_j}
           \\
           L_t^{u_{j+1}-u_j+u_1-h} \diff{u_{j+1}}
           \dots
           L_t^{u_3-u_2+u_1-h} \diff{u_{j_3}}
           L_t^{u_2-u_1+u_1-h} \diff{u_{j_2}}
           \\
           J_{q-1-j}(u_1-h,u_{j+2})J_{q-2-j}(u_1-h,u_{j+2})
           \diff{M_{u_{j+2}}}
         \end{multlined}
     \\ & \qquad \qquad+ R_h,
  \end{align*}
  where $\Ex{\abs{R_h}}=o(h^{q+1})$.
  Consequently, by another
  application of stochastic Fubini, 
  \begin{align} \notag
   \MoveEqLeft \int_{-\infty}^x K_j(U_{j,q-1,u_1},u_1-h,u_1) L_t^{u_1} \diff{u_1}
    \\ &= \notag
    \begin{multlined}[t][.8\columnwidth]
           \int_{-\infty}^x
           \int_{u_1-h}^{u_1}
           \int_{u_{j+2}}^{u_1} \int_{u_{j+2}}^{u_2} \dots \int_{u_{j+2}}^{u_j}
           \\
           L_t^{u_{j+1}-u_j+u_1-h} \diff{u_{j+1}}
           \dots
           L_t^{u_3-u_2+u_1-h} \diff{u_{j_3}}
           L_t^{u_2-u_1+u_1-h} \diff{u_{j_2}}
           \\
           J_{q-1-j}(u_1-h,u_{j+2})J_{q-2-j}(u_1-h,u_{j+2})
           \diff{M_{u_{j+2}}}
           L_t^{u_1}
           \diff{u_1}
         \end{multlined}
    \\ &\qquad \qquad           + R_h \notag
    \\ &= \label{eq:6}
    \begin{multlined}[t][.8\columnwidth]
           \int_{-\infty}^x
           \int_{u_{j+2}}^{(u_{j+2}+h)\land x}
           \int_{u_{j+2}}^{u_1} \int_{u_{j+2}}^{u_2} \dots \int_{u_{j+2}}^{u_j}
           \\
           L_t^{u_{j+1}-u_j+u_1-h} \diff{u_{j+1}}
           \dots
           L_t^{u_3-u_2+u_1-h} \diff{u_{j_3}}
           L_t^{u_2-u_1+u_1-h} \diff{u_{j_2}}
           \\
           J_{q-1-j}(u_1-h,u_{j+2})J_{q-2-j}(u_1-h,u_{j+2})
                      L_t^{u_1}
           \diff{u_1}
           \diff{M_{u_{j+2}}}
         \end{multlined}
               \\ &\qquad \qquad + \widetilde{R}_h, \notag
  \end{align}
  where again, $\Ex{\abs{\widetilde{R}_h}}=o(h^{q+1})$.
  From
  here, a tedious but straightforward application of Burkholder-Davis-Gundy,
  Jensen's inequality and Lemma~\ref{lem:jq}, in the same way as we have done to
  treat 
  the case $q=0$, yields that the $L^1$-norm of the iterated
  integral~\eqref{eq:6} is of order $o(h^{q+1+\varepsilon})$ for any
  $\varepsilon \in  
  (0,1/2)$.  

  \textit{Step 2.}
  As $\abs{X_{r,u}(v)} \leq C \, (v-u)$, Lemma~\ref{lem:jq} implies that
  \begin{equation*}
    \norm{K_{q-1}(X_{r,u-h}^2,u-h,u)}_2 \leq C h^{q+1}
    \norm{L_t^{\ast}}_{2^{q}}^{q-1}
  \end{equation*}
  and thus, by Fubini,
  \begin{align*}
    \MoveEqLeft[9]
    \E \left[ \abs{
    \int_{-\infty}^x K_{q-1}(X_{r,u-h}^2,u-h,u) \diff{\left\langle M,M
    \right\rangle_u}}
    \right]
   \\  &=4
      \int_{-\infty}^x
      \E \left[ \abs{K_{q-1}(X_{r,u-h}^2,u-h,u) L_t^u} \right]
      \diff{u}
    \\ &\leq
      4
      \int_{-\infty}^x
      \norm{K_{q-1}(X_{r,u-h}^2,u-h,u)}_2 \norm{L_t^u}_2
      \diff{u}
    \\ &\leq
      C h^{q+1} \norm{L_t^{\ast}}_{2^q}^{q-1} \int_{-\infty}^x
         \norm{L_t^u}_2 
         \diff{u}.
  \end{align*}
  Reasoning as in Step 1, we see that replacing one of the local times 
  inside $K_q$ or $X_{r,u-h}$ by a difference $L_t^u-L_t^{u-a}$ such that
  $\abs{a}<h$ has the  
  effect of replacing one power of $\norm{L_t^{\ast}}$ by $h^{\varepsilon}$,
  and consequently introducing an additional, negligible summand. 
  To exhibit the asymptotic behaviour of
  \begin{align}
    \notag
    \MoveEqLeft K_{q-1}(X_{r,u_1-h}^2,u_1-h,u_1)
    \\ &= \notag
         \begin{multlined}[t][.85\columnwidth]
             \int_{u_1-h}^{u_1} \dots
    \int_{u_1-h}^{u_{q-2}} 
         \int_{u_1-h}^{u_{q-1}}
         \left(
           \int_{u_1-h}^{u_{q}}
           Y_{u_{q+1}}
        \diff{u_{q+1}}\right)^2
  \\
         \diff{\left\langle M,M \right\rangle_{u_q}}
         \diff{\left\langle M,M \right\rangle_{u_{q-1}}}
         \dots
         \diff{\left\langle M,M \right\rangle_{u_2}}
         \end{multlined}
    \\ &= \label{eq:65}
         \begin{multlined}[t][.85\columnwidth]
               4^{q-1}
    \int_{u_1-h}^{u_1} \dots
    \int_{u_1-h}^{u_{q-2}} 
         \int_{u_1-h}^{u_{q-1}}
         \left(
    \int_{u_1-h}^{u_{q}}
    Y_{u_{q+1}} \diff{u_{q+1}} \right)^2
  \\
    L_t^{u_{q}} \diff{u_{q}}
    L_t^{u_{q-1}} \diff{u_{q-1}}
    \dots
    L_t^{u_{2}} \diff{u_{2}},
         \end{multlined}
  \end{align}
  note that for two real numbers $a$ and $b$, it holds that
 \begin{equation*}
   a^r - b^r = \sum_{k=1}^{r} a^{r-k} (a-b) b^{k-1}.
 \end{equation*}
 Thus, setting $a=Y_{u_{q+1}}$, $b=Y_{u_1}$ and exploiting the H\"older
 continuity of $Y$, we can replace the innermost 
 integral $\int_{u_1-h}^{u_q} Y_{u_{q+1}} \diff{u_{q+1}}$ in~\eqref{eq:65}
 by $Y_{u_1} (u_{q+1}-u_1+h)$, at the cost of introducing negligible 
 summands. In formulas, up to negligible summands, the right hand side of~\eqref{eq:65} is
 equal to 
 \begin{equation*}
   4^{q-1}
   Y_{u_1}^2
   \int_{u_1-h}^{u_1} \dots
    \int_{u_1-h}^{u_{q-2}} 
    \int_{u_1-h}^{u_{q-1}}
    (u_{q+1}-u_1+h)^2
    L_t^{u_{q}} \diff{u_{q}}
    L_t^{u_{q-1}} \diff{u_{q-1}}
    \dots
    L_t^{u_{2}} \diff{u_{2}}.
 \end{equation*}
 Repeating this procedure, iteratively replacing $L_t^{u_q}$, $L_t^{u_{q-1}}$,
 etc. and evaluating the resulting purely deterministic integral, we see that 
  \begin{equation*}
    K_{q-1}(X_{u_1-h}^2,u_1-h,u_1)
    =
    \frac{2^{2q-1}}{(q+1)!} h^{q+1}
    (L_t^{u_1})^{q-1} Y_{u_1}^2 + o(h^{q+1}).
  \end{equation*}
  Consequently
  \begin{equation*}
    \frac{1}{h^{q+1}} \int_{-\infty}^x K_{q-1}(X_{u-h}^2,u-h,u) \diff{\left\langle M,M
      \right\rangle_{u}}
    =
    \frac{2^{2q+1}}{(q+1)!}
    \int_{-\infty}^x \left( L_t^u \right)^{q} Y_u^2 \diff{u} + o(1),
  \end{equation*}
  concluding the proof.
\end{proof}

\section{Main result}
\label{s:main}

We now have all necessary tools at our disposal to prove Theorem~\ref{thm:4}
stated in the introduction.
We will proceed in two  steps. First, in the forthcoming Theorem~\ref{thm:1}, we
will prove a limit theorem for a certain iterated integral, which, in the proof
of Theorem~\ref{thm:4}, will turn out to be the leading term when applying the
Kailath-Segall identity. 

\begin{theorem}
  \label{thm:1}
  Let $M_x$ be the local martingale part of Brownian local time $L_t^x$ and
  $I_q$ be the iterated integrals with respect to $M_x$. Then, for any integer
  $q \geq 2$ it holds that 
  \begin{equation}
    \label{eq:27}
    \frac{q!}{h^{(q+1)/2}}
    \int_{-\infty}^{\infty} I_q(x,x+h) \diff{x}
    \to
    c_q
    \sqrt{
      \int_{-\infty}^{\infty} \left( L_t^x \right)^{q} \diff{x}
    } \, Z, 
  \end{equation}
  where $Z \sim \mathcal{N}(0,1)$, independent of $(L_t^{x})_{x \in \R}$ and the
  constant 
  $c_q$ is given by
  \begin{equation*}
     c_q = \frac{2^{2q+1} q!}{q+1}.
  \end{equation*}
\end{theorem}

\begin{proof}
  By definition, we have that
  \begin{equation}
    \label{eq:53}
    \int_{-\infty}^{\infty} I_q(x,x+h) \diff{x}
    =
    \int_{-\infty}^{\infty} \int_{-\infty}^{\infty}
    1_{(x,x+h)}(u)  I_{q-1}(x,u) \diff{M_u} \diff{x}. 
  \end{equation}
  If we set
  \begin{equation*}
    \phi(x,u) = 1_{(x,x+h)}(u)  I_{q-1}(x,u),
  \end{equation*}
  then, by the deterministic Fubini theorem, Cauchy-Schwarz, Jensen's inequality
  and Lemma~\ref{lem:jq} (recall that $I_q(x,y)=J_q(1,x,y)$), 
  \begin{align*}
    \MoveEqLeft[4]
    \E \left[
      \int_{-\infty}^{\infty}
      \int_{-\infty}^{\infty} \phi(x,u)^2 \diff{x}
      \diff{\left\langle M,M \right\rangle_u}
    \right]
    \\ &= 4
         \E \left[ \int_{-\infty}^{\infty} \int_{u-h}^{u} I_{q-1}(x,u)^2
         \diff{x} L_t^u \diff{u} \right] 
    \\ &\leq
          4
          \int_{-\infty}^{\infty} \norm{ \int_{u-h}^{u} I_{q-1}(x,u)^2
          \diff{x}}_2 \norm{L_t^u}_2 \diff{u}
    \\ &\leq
          4
          \int_{-\infty}^{\infty} \norm{ \int_{u-h}^{u} I_{q-1}(x,u)^2
          \diff{x}}_2 \norm{L_t^u}_2 \diff{u}
    \\ &\leq
         C \int_{-\infty}^{\infty} \norm{L_t^u}_2 \diff{u}
         < \infty.
  \end{align*}
  This shows that we can apply the stochastic Fubini theorem to the right
  hand side of~\eqref{eq:53} and also that 
  \begin{equation*}
    \int_{-\infty}^{\cdot} \int_{-\infty}^{\infty} \phi(x,u) \diff{x} \diff{M_u}
  \end{equation*}
  is a square integrable martingale (which in particular posesses a limit).
  Therefore, we get 
  \begin{align*}
    \int_{-\infty}^{\infty} I_q(x,x+h) \diff{x}
    &=
    \int_{-\infty}^{\infty} \int_{u_1-h}^{u_1}
    I_{q-1}(x,u_1) \diff{M_{u_1}} \diff{x}
    \\ &=
    \int_{-\infty}^{\infty} \int_{u_1-h}^{u_1}
         \int_x^{u_1} I_{q-2}(x,u_2) \diff{M_{u_2}}
         \diff{x}
         \diff{M_{u_1}}.
  \end{align*}
  Iterating this procedure another $q-1$-times, we obtain
  \begin{align*}
    \int_{-\infty}^{\infty} I_q(x,x+h) \diff{x}
    &=
    \int_{-\infty}^{\infty} \int_{u_1-h}^{u_1} \dots \int_{u_1-h}^{u_{q-1}}
         \int_{u_1-h}^{u_q}
         \diff{x} \diff{M_{u_q}} \dots \diff{M_{u_2}} \diff{M_{u_1}}
    \\ &=
         \int_{-\infty}^{\infty} J_{q-1}(X_{u},u-h,u) \diff{M_u},
  \end{align*}  
  where
  \begin{equation*}
    X_{u}(v) = v-u+h.
  \end{equation*}
  By Lemma~\ref{lem:jq}, the process $(\widetilde{M}_x^h)_{x \geq -\infty}$,
  defined  by
  \begin{equation*}
    \widetilde{M}_{x}^h = \frac{q!}{h^{(q+1)/2}} \int_{-\infty}^x
    J_{q-1}(X_{u}, u-h,u) \diff{M_u}, 
  \end{equation*}
  is a $L^p$-bounded, uniformly integrable martingale for any $h > 0$, which
  vanishes at $-\infty$. Moreover,
  Lemma~\ref{lem:3} yields that for $x_0 > -\infty$
  \begin{equation*}
    \sup_{ x \in (-\infty,x_0]} \abs{\left\langle \widetilde{M}^h, M
      \right\rangle_x} \to 0 
  \end{equation*}
  and Lemma~\ref{lem:4} shows that for $x \in \R \cup \Set{-\infty,\infty}$, 
  \begin{equation*}
         \left\langle \widetilde{M}^h,\widetilde{M}^h \right\rangle_x \to c_q^2
         \int_{-\infty}^x \left( L_t^u \right)^q \diff{u},
  \end{equation*}
  where both convergences hold in $L^1$ (and we define $\int_{-\infty}^{-\infty}
  f(u) \diff{u} = 0$). 
  Consequently, the asymptotic Ray-Knight Theorem~\ref{thm:ark} implies that 
  \begin{equation}
    \label{eq:31}
    \widetilde{M}_x^h \xrightarrow{d} c_q \sqrt{\int_{-\infty}^x
      \left(L_t^x\right)^{q} \diff{x}} \, Z
  \end{equation}
   for $h \to 0$, where $Z \sim \mathcal{N}(0,1)$, independent of $(M_x)_{x
     \in \R}$
   (and thus also of $(L_t^x)_{x \in \R}$ and the underlying Brownian motion).
   Indeed, if $\beta$ and $\beta^h$ denote the Dambis-Dubins-Schwarz Brownian
   motions of $M$ and $\widetilde{M}^h$, respectively,
   Theorem~\ref{thm:ark} yields that
   \begin{equation*}
      \left(\beta,\beta^h,\left\langle \widetilde{M}^h,\widetilde{M}^h
      \right\rangle \right)
      \xrightarrow{d}
      \left(\beta,\widetilde{\beta},c_q^2 
     \int_{-\infty}^{\cdot} (L_t^u)^{q} \diff{u} \right),
   \end{equation*}
where
   $\widetilde{\beta}$ is a standard 
   Brownian motion which is independent of $\beta$. Consequently,
   \begin{equation*}
        \widetilde{M}^h_x = \beta^h_{\left\langle 
            \widetilde{M}^h,\widetilde{M}^h \right\rangle_x}
        \xrightarrow{d}
        \widetilde{\beta}_{c_q^2 
     \int_{-\infty}^x (L_t^u)^q \diff{u}}.
   \end{equation*}
   Letting $x$ tend to infinity finishes the proof.
\end{proof}

We now turn to the proof of Theorem~\ref{thm:4} from the introduction, which we
restate here for convenience.

\begin{theorem}
  \label{thm:2}
  For integers $q \geq 2$ it holds that
  \begin{equation*}
    \frac{1}{h^{\frac{q+1}{2}}}
    \left(
      \int_{-\infty}^{\infty} \left( L_t^{x+h} - L_t^x \right)^q 
      \diff{x} 
      +
      R_{q,h} \right)
    \xrightarrow{d}
    c_q
    \sqrt{
       \int_{-\infty}^{\infty} (L_t^x)^{q} \diff{x}
    } \, Z,
  \end{equation*}
  where $Z$ is a standard Gaussian random variable, independent of
  $(L_t^x)_{x \in \R}$, the random variable $R_{q,h}$ is given by
  \begin{equation*}
    R_{q,h} = 
          \sum_{k=1}^{\lfloor \frac{q}{2} \rfloor} a_{q,k} 
      \int_{-\infty}^{\infty} \left( L_t^{x+h} - L_t^x \right)^{q-2k} \left(
        4 \int_x^{x+h} L_t^u \diff{u} \right)^k \diff{x}
  \end{equation*}
  and the constants $a_{q,k}$ and $c_q$ are defined as
  \begin{equation*}
    a_{q,k} = \frac{(-1)^k q!}{2^k k! (q-2k)!} \qquad \text{and} \qquad
    c_q = \sqrt{\frac{2^{2q+1} q!}{q+1}}.
  \end{equation*}
\end{theorem}

\begin{proof}
  Let $L_t^x = M_x + V_x$ be the canonical semimartingale decomposition of
  Brownian local time (we suppress the dependence of the fixed parameter $t$ for
  brevity) and $I_q$ the iterated integrals with respect to the local martingale
  $M_x$.
  Throughout the proof, we use the shorthand notation $I_q^h(x)=I_q(x,x+h)$,
  $\Delta_x^h L_t^x = L_t^{x+h} - L_t^x$ and $\Delta_x^h
  V_{t,x} = V_{t,x+h}- V_{t,x}$.
By the binomial theorem and the fact that $M_{x+h}-M_x=I_1^h(x)$, we get that 
\begin{equation}
  \label{eq:40}
  \left(\Delta_x^h L_t^x \right)^q = \left( I_1^h(x) + \Delta_x^h V_{t,x}
  \right)^q 
  =
  \sum_{k=0}^q \binom{q}{k} \left( I_1^h(x) \right)^{q-k} \left( \Delta_x^h
    V_{t,x} \right)^k.
\end{equation}
As by Burkholder-Davis-Gundy
\begin{align*}
  \norm{ I_1^h(x)^{q-2k}}_p
  &\leq
  C_p
  \norm{
  \left(
  4 \int_x^{x+h} L_t^u \diff{u}
  \right)^{1/2}
  }_{(q-2k)p}^{q-2k}
  \\ &\leq
  C_p h^{(q-2k)/2} \norm{ L_t^{\ast}}_p^{(q-2k)/2}
\end{align*}
and, writing $\Delta_x^h V_{x} = \int_x^{x+h} A_{u} \diff{u}$, 
\begin{equation*}
  \norm{\left(\frac{\Delta_x^h V_{x}}{h}\right)^k }_p
  \leq
       \norm{\left( \frac{1}{h} \int_x^{x+h} \abs{A_{u}} \diff{u}
       \right)^{k}}_p
  \xrightarrow{h \to 0}
      \norm{A_{u}}_{kp}^k
\end{equation*}
we see that
\begin{equation*}
  h^{-(q+k)/2} \E  \left[\int_{-\infty}^{\infty} \abs{I_1^h(x)}^{q-k}
    \abs{\Delta_x^h V_{x}}^k \diff{x} \right]
  < \infty  
\end{equation*}
Thus, all those summands in
the sum on the right hand side of~\eqref{eq:40} for which $k > 1$ do not
contribute to the limit. To be more precise, it 
holds that 
\begin{multline}
  \label{eq:29}
  \int_{-\infty}^{\infty} \left( \Delta_x^h L_t^x 
  \right)^q  \diff{x}
  \\ =   
  \int_{-\infty}^{\infty} \left( I_1^h(x) \right)^q \diff{x} +
  q
  \int_{-\infty}^{\infty} \left( I_1^h(x) \right)^{q-1} (\Delta_x^h V_{x})
  \diff{x}
  + R_{1,h},
\end{multline}
where $R_{1,h}/h^{(q+1)/2}$ converges to zero in $L^p$ for $h \to 0$. 
The Kailath-Segall identity~\eqref{eq:1} and another application of the binomial
theorem yields
\begin{equation}
  \label{eq:21}
  \begin{split}
  q! I_q^h(x) - \left( I_1^h(x) \right)^q
  &=
       \sum_{k=1}^{\lfloor \frac{q}{2} \rfloor} a_{q,k}
       \left( I_1^h(x) \right)^{q-2k}
       \left( 4 \int_x^{x+h} L_t^u \diff{u}  \right)^k
  \\ &=
       \sum_{k=1}^{\lfloor \frac{q}{2} \rfloor} a_{q,k}
       \left( \Delta_x^h L_t^x - \Delta_x^h V_{t,x}\right)^{q-2k}
       \left( 4 \int_x^{x+h} L_t^u \diff{u}  \right)^k
       \\ &=
  \delta_{\lfloor \frac{q}{2} \rfloor, \frac{q}{2}} \, 
  a_{\frac{q}{2},k} \left( 4 \int_x^{x+h} L_t^u \diff{u} \right)^{q/2}
  \\ &\qquad + 
       \sum_{k=1}^{\lfloor \frac{q-1}{2} \rfloor} \sum_{j=0}^{q-2k} (-1)^j a_{q,k}
       \binom{q-2k}{j}
  \\ &\qquad \qquad
       \left( \Delta_x^h L_t^x \right)^{q-2k-j} \left(\Delta_x^{h} V_{t,x}
       \right)^j \left( 4 \int_x^{x+h} L_t^u \diff{u}  \right)^k, 
\end{split}  
\end{equation}
where $\delta_{\lfloor \frac{q}{2} \rfloor, \frac{q}{2}}=1$ if $q$ is even and
zero otherwise. By the H\"older continuity property of the Brownian local
time, $\Delta_x^h L_t^x < h^{\varepsilon}$ for any $ \varepsilon \in (0,1/2)$.
Therefore, a similar argument as above shows that if we first
multiply~\eqref{eq:21} by $(L_t^x)^r$ and then integrate on both sides, all
summands in the 
double sum for which  
$\frac{q-2k-j}{2} + k + j > \frac{q+1}{2}$, i.e. for which $j>1$ do not 
contribute to the limit. In formulas, we have
\begin{equation}
  \label{eq:20}
\begin{split}
  q!
  \int_{-\infty}^{\infty} I_q^h(x)  &\diff{x}
  -
  \int_{-\infty}^{\infty} \left( I_1^h(x) \right)^q  \diff{x}
  \\ & =  
  \sum_{k=1}^{\lfloor \frac{q}{2} \rfloor} a_{q,k}
       \int_{-\infty}^{\infty}
       \left( \Delta_x^h L_t^x \right)^{q-2k} \left( 4 \int_x^{x+h} L_t^u
         \diff{u}  \right)^k  
       \diff{x} 
  \\ &\qquad  -
  \sum_{k=1}^{\lfloor \frac{q-1}{2} \rfloor} a_{q,k}
  (q-2k)
  \\ & \qquad \qquad \qquad
  \int_{-\infty}^{\infty}
       \left( \Delta_x^h L_t^x \right)^{q-2k-1} \Delta_x^h V_{x} \left( 4
         \int_x^{x+h} L_t^u \diff{u}  \right)^k
       \diff{x}
  \\ & \qquad 
       + R_{2,h},
\end{split}  
\end{equation}
where $R_{2,h}/h^{(q+1)/2}$ converges to zero in $L^p$ for $h \to
0$. Analogously, we derive that
\begin{equation}
  \label{eq:19}
  \begin{multlined}[c][.85\columnwidth]
    q!
    \int_{-\infty}^{\infty} I_{q-1}^h(x) (\Delta_x^h V_{x}) \diff{x}
    -
    q \int_{-\infty}^{\infty} \left( I_1^h(x) \right)^{q-1} \Delta_x^h V_{x}
    (L_t^x)^r
    \diff{x}
    \\ = q
    \sum_{k=1}^{\lfloor \frac{q-1}{2} \rfloor}
    a_{q-1,k}
    \int_{-\infty}^{\infty}
    \left( \Delta_x^h L_t^{x} \right)^{q-1-2k}
    \Delta_x^h V_{x}
    \left( 4 \int_x^{x+h} L_t^u \right)^k
    \diff{x}
    + R_{3,h},
  \end{multlined}
\end{equation}
where $R_{3,h}/h^{(q+1)/2}$ converges to zero in $L^p$ for $h \to 0$. If
we now plug~\eqref{eq:20} and~\eqref{eq:19} into~\eqref{eq:29} and exploit the
identity $qa_{q-1,k} - (q-2k) a_{q,k}=0$ for $q \geq 2$ and $1 \leq k \leq \lfloor
\frac{q-1}{2} \rfloor$ (which can be shown by straightforward induction), we
obtain that
\begin{align*}
      \int_{-\infty}^{\infty} \left( \Delta_x^h L_t^x \right)^q &
      \diff{x}
      &=
      \begin{aligned}[t]
  q! &\int_{-\infty}^{\infty} I_q^h(x) \diff{x}
  \\  &-
  \sum_{k=1}^{\lfloor \frac{q}{2} \rfloor} a_{q,k} \int_{-\infty}^{\infty}
  \left( \Delta_x^h 
    L_t^x\right)^{q-2k} \left( 4 \int_x^{x+h} L_t^u \diff{u} \right)^k 
  \diff{x} 
  \\  &+
  q! \int_{-\infty}^{\infty} I_{q-1}^h(x) \Delta_x^h V_{x} \diff{x}
  \\ &+ \widetilde{R}_h,        
      \end{aligned}
\end{align*}
where $\widetilde{R}_h/h^{(q+1)/2}$ converges to zero in $L^p$ for $h \to
0$. By Theorem~\ref{thm:1}, the proof is finished if we can show that
\begin{equation}
  \label{eq:26}
  \frac{1}{h^{(q+1)/2}}
  \int_{-\infty}^{\infty} I_{q-1}^h(x) \Delta_x^h V_{x} \diff{x}
  \xrightarrow{d}  0.
\end{equation}
  For $V_{x} = \int_{-\infty}^{x}
  A_{u} \diff{u}$, we write
  \begin{multline}
    \label{eq:25}
    \frac{1}{h^{(q+1)/2}} \int_{-\infty}^{\infty} I_{q-1}^h(x) \Delta_x^h
    V_{x} \diff{x}
    \\ =\frac{1}{h^{(q-1)/2}} \int_{-\infty}^{\infty} I_{q-1}^h(x) A_{x}
    \diff{x}
    +
    \int_{-\infty}^{\infty} \frac{I_{q-1}^h(x)}{h^{(q-1)/2}} \left(
      \frac{1}{h} \int_x^{x+h} A_{u} 
    \diff{u} - A_{x} \right)  \diff{x}.
\end{multline}
For the first integral on the right hand side of~\eqref{eq:25},
stochastic  Fubini yields that 
  \begin{align*}
    \frac{1}{h^{(q-1)/2}}
    \int_{-\infty}^{\infty} I_{q-1}^h(x) A_{t,x}  \diff{x}
     &=
    \frac{1}{h^{(q-1)/2}}
    \int_{-\infty}^{\infty} \int_x^{x+h} I_{q-2}(x,u) \diff{M_u} A_{x} 
    \diff{x}
    \\ &=
    \frac{1}{h^{(q-1)/2}}
    \int_{-\infty}^{\infty} \int_{u-h}^{u} I_{q-2}(x,u) A_x  \diff{x}
         \diff{M_u}.
  \end{align*}
  Thus, by Burkholder-Davis-Gundy and Jensen's inequality
  \begin{align} \notag
    \MoveEqLeft[4]
    \norm{ \frac{1}{h^{(q-1)/2}} \int_{-\infty}^{\infty} \int_{u-h}^{u}
     I_{q-2}(x,u) 
     A_{x}  \diff{x}
     \diff{M_u}}_2
    \\ &\leq C \notag
     \norm{ \left( \int_{-\infty}^{\infty} \left( \frac{1}{h^{(q-1)/2}}
         \int_{u-h}^u 
         I_{q-2}(x,u) A_{x}  \diff{x} \right)^2
         L_t^u \diff{u} \right)^{1/2}}_2
    \\ &\leq
         C \notag
         \left( \int_{-\infty}^{\infty}
         \E \left[  \left( \frac{1}{h^{(q-1)/2}}
         \int_{u-h}^u 
         I_{q-2}(x,u) A_{x}  \diff{x} \right)^2
         L_t^u
          \right]
         \diff{u} \right)^{1/2}
    \\ &\leq
         \label{eq:62}
         C
         \left(
         \frac{1}{h^{q-1}}
         \int_{-\infty}^{\infty}
         \norm{ 
         \int_{u-h}^u 
         I_{q-2}(x,u) A_{x}  \diff{x}}_4^2
         \norm{L_t^u}_2
         \diff{u} \right)^{1/2}.
  \end{align}
  By Cauchy-Schwarz, Jensen and Lemma~\ref{lem:jq}, it follows that
  \begin{align*}
    \MoveEqLeft[4]
    \norm{ \int_{u-h}^u I_{q-2}(x,u) A_x  \diff{x}}_4
    \\  &\leq 
       \norm{
         \left(\int_{u-h}^u I_{q-2}(x,u)^2 \diff{x} \right)^{1/2}
         \left(\int_{u-h}^u A_x^2 \diff{x} \right)^{1/2} }_4
    \\ &\leq
       \norm{
         \int_{u-h}^u I_{q-2}(x,u)^2 \diff{x}}_{8}^{1/2}
         \norm{
         \int_{u-h}^u A_x^2 \diff{x}}_{8}^{1/2}
    \\ &\leq
         \left(
         h^{7}
         \int_{u-h}^u \norm{I_{q-2}(x,u)}_{16}^{16} \diff{x}
         \right)^{1/16}
         \norm{
         \int_{u-h}^u A_x^2 \diff{x}}_{8}^{1/2}
    \\ &\leq
         C
         h^{(q-1)/2} \norm{L_t^{\ast}}_{2^{q+2}}^{(q-2)/2}
         \norm{
         \int_{u-h}^u A_x^2 \diff{x}}_{8}^{1/2}, 
  \end{align*}
  which, plugged into~\eqref{eq:62}, yields
  \begin{multline*}
     \norm{ \frac{1}{h^{(q-1)/2}} \int_{-\infty}^{\infty} \int_{u-h}^{u}
     I_{q-2}(x,u) 
     A_{x}  \diff{x}
     \diff{M_u}}_2
   \\ \leq
   C
   \norm{L_t^{\ast}}_{2^{q+2}}^{(q-2)/2}
   \left(
   \int_{-\infty}^{\infty}
   \norm{\int_{u-h}^u A_x^2 \diff{x}}_8 \norm{L_t^u}_2 \diff{u}
   \right)^{1/2}
  \end{multline*}
  and by the Vitali convergence theorem (see for
  example~\cite[p.133]{rudin_real_1987}), the integral on the right hand side  
  converges to zero.
  
  Let us turn to the second integral on the right hand side
  of~\eqref{eq:25}. By deterministic Fubini,
  \begin{align}
    \notag
    \MoveEqLeft[8]
    \E
      \left[
      \abs{
      \int_{-\infty}^{\infty}
      \frac{I_{q-1}^h(x)}{h^{(q-1)/2}}
    \left(
      \frac{1}{h} \int_x^{x+h} A_u \diff{u} - A_x
      \right)
      \diff{x}}
      \right]
    \\ &\leq \notag
      \int_{-\infty}^{\infty}
      \E \left[
      \abs{
      \frac{I_{q-1}^h(x)}{h^{(q-1)/2}}
    \left(
      \frac{1}{h} \int_x^{x+h} A_u \diff{u} - A_x
      \right)
      }
      \right]
    \\ &\leq \label{eq:70}
            \int_{-\infty}^{\infty}
      \norm{\frac{I_{q-1}^h(x)}{h^{(q-1)/2}}}_2
      \norm{
      \frac{1}{h} \int_x^{x+h} A_u \diff{u} - A_x
      }_4
      \diff{x}.
  \end{align}
  Lemma~\ref{lem:2} gives that
  \begin{equation*}
    \norm{\frac{I_{q-1}^h(x)}{h^{(q-1)2}}}_2
    \leq \norm{\frac{1}{h} \int_x^{x+h} \left\langle M,M
      \right\rangle_u}_{2(q-1)}^{(q-1)/2}
    =
    4^{(q-1)/2} \norm{\frac{1}{h} \int_x^{x+h} L_t^u
      \diff{u}}_{2(q-1)}^{(q-1)/2}
  \end{equation*}
  and the right hand side converges to $4^{(q-2)/2}
  \norm{L_t^x}_{2(q-1)}^{(q-1)/2}$. As by Lemma~\ref{lem:5} it holds that
  \begin{equation*}
    \int_{-\infty}^{\infty}
    \norm{L_t^x}_{2(q-1)}^{(q-1)/2}\diff{x} < \infty,
  \end{equation*}
  the Vitali convergence theorem implies that the integral on the right
  hand side of~\eqref{eq:70} converges to zero as well, concluding the proof. 
\end{proof}

\begin{remark}
  \label{rem:4}
  In the following, we continue using the shorthand $\Delta_x^hL_t^x = L_t^{x+h}
  - L_t^x$. 
  \begin{enumerate}[1.]
  \item For $q=2$, Theorem~\ref{thm:2} reads
    \begin{multline*}
      \frac{1}{h^{3/2}}
      \left(
      \int_{-\infty}^{\infty} (\Delta_x^h L_t^x)^2 \diff{x}
      -
      4
      \int_{-\infty}^{\infty} \int_x^{x+h} L_t^u \diff{u} \diff{x}
      \right) \\
      \xrightarrow{d}
      \sqrt{
        \frac{64}{3} \int_{-\infty}^{\infty} (L_t^x)^2 \diff{x}
      } \, Z,
    \end{multline*}
    and as by Fubini and the occupation times formula
    \begin{equation*}
      4\int_{-\infty}^{\infty} \int_x^{x+h} L_t^u \diff{u} \diff{x}
      =
      4 \int_{-\infty}^{\infty} \int_{u-h}^{u} \diff{x} L_t^u \diff{u}
      =
      4ht,
    \end{equation*}
    we recover the second order result~\eqref{eq:9} from~\cite{chen_clt_2010}.
  \item For $q=3$, Theorem~\ref{thm:2} reads
    \begin{multline*}
      \frac{1}{h^2}
      \left(
        \int_{-\infty}^{\infty} (\Delta_x^h L_t^x)^3 \diff{x}
        - 12 \int_{-\infty}^{\infty} \Delta_x^h L_t^x \int_x^{x+h} L_t^u
        \diff{u} \diff{x}
      \right) \\
      \xrightarrow{d} \sqrt{192 \int_{-\infty}^{\infty} (L_t^x)^3 \diff{x}} \,
      Z 
    \end{multline*}
    and as
    \begin{equation}
      \label{eq:15}
      \int_{-\infty}^{\infty}
      \Delta_x^h L_t^x \int_x^{x+h} L_t^u
      \diff{u} \diff{x}
      =
      \lim_{n \to \infty}
      \int_{-n}^{n}
      \frac{\mathrm{d}}{\mathrm{d} x}
      \left( \int_x^{x+h} L_t^u \diff{u} \right)^2
      \diff{x}
      = 0,
    \end{equation}
    we recover the third order result~\eqref{eq:32} from~\cite{rosen_clt_2011}. 
  \item For $q=4$, Theorem~\ref{thm:2} becomes
    \begin{multline}
      \label{eq:18}
      \frac{1}{h^{5/2}}
      \left(
        \int_{-\infty}^{\infty} (\Delta_x^h L_t^x)^4 \diff{x}
        - 24 \int_{-\infty}^{\infty} (\Delta_x^h L_t^x)^2 \int_x^{x+h} L_t^u
        \diff{u} \diff{x} \right.
      \\ \left.
        + 48 \int_{-\infty}^{\infty} \left( \int_x^{x+h} L_t^u \diff{u}
        \right)^2 \diff{x} 
      \right)
      \xrightarrow{d} c_4 \sqrt{\int_{-\infty}^{\infty} (L_t^x)^4 \diff{x}} \,
      Z.
    \end{multline}
    Compared to Rosen's Conjecture~\ref{conj:4}, we see that our compensator
    differs 
    from the conjectured 
    \begin{equation*}
      - 24h \int_{-\infty}^{\infty} (\Delta_x^h L_t^x)^2 L_t^x \diff{x}
      + 48h^2 \int_{-\infty}^{\infty} (L_t^x)^2 \diff{x}
      - \int_{-\infty}^{\infty} (\Delta_x^h L_t^x) L_t^x \diff{x}.
    \end{equation*}
    In view of~\eqref{eq:15}, we would recover this conjectured compensator
    from~\eqref{eq:18} if
    we could replace the term $\frac{1}{h} \int_x^{x+h} L_t^u \diff{u}$ by its
    limit $L_t^x$. However, as by the Mean Value Theorem
    \begin{equation*}
      \abs{
      \frac{1}{h} \int_x^{x+h} L_t^u \diff{u} - L_t^x}
      \leq h^{\varepsilon}
    \end{equation*}
    for any $\varepsilon \in (0,1/2)$, but be would need an order greater than
    $h^{1/2}$ to do the replacement, proving that our
    compensator is equal to the conjectured one (up to negligible terms) does
    not seem to be straightforward.
  \item It is natural to ask whether, as in the cases $q=2$ and $q=3$, a
    central limit theorem continues to hold for $q \geq 4$. It turns out that
    this is equivalent to asking whether $R_{q,h}$ can be replaced by its
    expectation in the statement of Theorem~\ref{thm:2}. Indeed, by inspecting
    the proof of Theorem~\ref{thm:2}, we see that the expectations of $R_{q,h}$ and
    $\int_{-\infty}^{\infty} \left( L_t^{x+h} - L_t^x \right)^q \diff{x}$ have
    the same order of convergence. Furthermore, if $q$ is odd, both of these
    expectations are zero by symmetry. Therefore, to obtain  a central limit
    theorem for $q \geq 4$, one has to show for odd $q \geq 5$ that
    $R_{q,h}/h^{(q+1)/2} \xrightarrow{d} 0$ and for even $q \geq 4$ that
\begin{equation*}
  \frac{1}{h^{(q+1)/2}} \left( R_{q,h} - \Ex{R_{q,h}} \right) \xrightarrow{d}
  0. 
\end{equation*}
Unfortunately, we have to leave this question open for further research.
  \end{enumerate}
\end{remark}

 Our space approach allows to generalize Theorem~\ref{thm:2} in several
directions. For 
example, in view of Lemmas~\ref{lem:3} and~\ref{lem:4}, a careful examination of
the proof of Theorem~\ref{thm:2} immediately yields the following result.

\begin{theorem}
  Let $L_t^x$ be Brownian local time and, for some $\alpha>0$, let $(Y_x)_{x \in
    \R}$ be a non-negative, uniformly bounded and almost surely
  $\alpha$-H\"older continous process which is adapted to   $(L_t^x)_{x \in
    \R}$.  
  Then, for integers $q \geq 2$, it holds that
  \begin{equation*}
    \frac{1}{h^{\frac{q+1}{2}}}
    \left(
      \int_{-\infty}^{\infty} \left( L_t^{x+h} - L_t^{x} \right)^q Y_x
      \diff{x}
      + \widetilde{R}_{t,h}
      \right)
    \xrightarrow{d}
    c_q
    \sqrt{
       \int_{-\infty}^{\infty} (L_t^x)^{q} Y_x^2 \diff{x}
    } \, Z,
  \end{equation*}
  where $Z$ is a standard Gaussian random variable, independent of
  $\left( (L_t^x)^qY_x^2\right)_{x \in \R}$,
  \begin{equation*}
    \widetilde{R}_{t,h} =
      \sum_{k=1}^{\lfloor \frac{q}{2} \rfloor} a_{q,k} 
      \int_{-\infty}^{\infty} \left( L_t^{x+h} - L_t^{x} \right)^{q-2k} \left(
        4 \int_x^{x+h} L_t^u \diff{u} \right)^k Y_x \diff{x}
  \end{equation*}
  and the constants $a_{q,k}$ and  $c_q$ are given by~\eqref{eq:8}.
\end{theorem}
As an explicit example, one can take $Y_x = (L_t^x)^r$ for any $r \geq 1$. 

Generalizing in another direction, we can replace the time variable $t$, which
is never touched in our proofs, with suitable stopping times $\tau$. A necessary
condition for such a stopping time 
$\tau$ is that $(L_{\tau}^x)_{x \in \R}$ admits a \emph{regular semimartingale
  decomposition} (see~\cite[Section 3]{peccati_hardys_2004}) on some interval $I$, by
which we mean the existence of a probability measure $Q$, a filtration 
$\Set{\mathcal{G}_x(\tau)}_{x \in I}$ and a $(\mathcal{G}_x(\tau),Q)$-Brownian
motion $(\beta_x)_{x \in I}$ such that $L_{\tau}^x$ is a
$(\mathcal{G}_x(\tau),Q)$-semimartingale with canonical decomposition
  \begin{equation*}
    L_{\tau}^x =
    \begin{cases}
      L_{\tau}^0 + 2 \int_0^{x} \sqrt{L_{\tau}^u} \diff{\beta_u} + \int_0^x A_{\tau,u}
      \diff{u} & \text{if $x \in I \cap \R_{+}$} \\
      L_{\tau}^0 - 2 \int_x^{0} \sqrt{L_{\tau}^u} \diff{\beta_u} - \int_x^0 A_{\tau,u}
      \diff{u} & \text{if $x \in I \cap \R_{-}$}.
    \end{cases}
  \end{equation*}

Again, a careful reevaluation of the proof of Theorem~\ref{thm:3} yields the
following set of sufficient conditions.

\begin{theorem}
  \label{thm:3}
 Let $L_t^x$ be the local time of Brownian motion and $\tau$ be a stopping time
 such that $(L_{\tau}^x)_{x \in \R}$ admits a regular semimartingale
 decomposition on some interval~$I$ with
 a finite variation kernel $A_{\tau,u}$  which satisfies 
  \begin{equation*}
  \int_{I} \abs{A_{t,x}}
    \diff{x} < \infty \qquad \text{and} \qquad
  \int_{I} \norm{A_{t,x}}_p \diff{x} < \infty
\end{equation*}
for $p \geq 1$. Furthermore, assume that $\norm{L_{\tau}^{\ast}}_p < \infty$
for $p \geq 1$. Then, for any positive integer $p \geq 2$, it holds that
  \begin{equation*}
    \frac{1}{h^{\frac{q+1}{2}}}
    \left(
      \int_{I} \left( L_{\tau}^{x+h} - L_{\tau}^x \right)^q 
      \diff{x} 
      +
      \widetilde{R}_{q,h} \right)
    \xrightarrow{d}
    c_q
    \sqrt{
       \int_{I} (L_{\tau}^x)^{q} \diff{x}
    } \, Z,
  \end{equation*}
  where $Z$ is a standard Gaussian random variable, independent of
  $(L_{\tau}^x)_{x \in \R}$, the random variable $\widetilde{R}_{q,h}$ is given
  by 
  \begin{equation*}
    \widetilde{R}_{q,h} = 
          \sum_{k=1}^{\lfloor \frac{q}{2} \rfloor} a_{q,k} 
      \int_{I} \left( L_{\tau}^{x+h} - L_{\tau}^x \right)^{q-2k}
      \left( 
        4 \int_x^{x+h} L_{\tau}^u \diff{u} \right)^k \diff{x}
  \end{equation*}
  and the constants $a_{q,k}$ and $c_q$ are defined in~\eqref{eq:8}.
\end{theorem}

An example of a stopping time verifying the conditions of Theorem~\ref{thm:3}
(with $I = \R_{+}$) comes from the Ray-Knight theorem (see~\cite[Ch. XI]{revuz_continuous_1999}):
If we take \begin{equation*} 
  \tau_0 = \inf \Set{t \geq 0 \colon L_t^0 > 0},
\end{equation*}
then  $L_{\tau_0}^x$ has a regular
semimartingale decomposition on $\R_+$ with 
finite variation kernel $A_{\tau_0,u}=0$. Furthermore, $(L_{\tau_0}^x)_{x \geq
  0}$ is equal in law to to a squared Bessel process started in zero with
dimension zero, and thus, for example by~\cite{yan_iterated_2005}, we have that 
$\norm{L_{\tau_0}^{\ast}}_p < \infty$.

\section{Acknowledgements}

I heartily want to thank Giovanni Peccati, to whom the idea to look
for alternative proofs of~\eqref{eq:9} and~\eqref{eq:32} through the
semimartingale decomposition of Brownian local time in space is due and who
provided valuable insight. I'm also indebted to Guillaume Poly for numerous
fruitful discussions which improved this paper in several ways. Finally, I would
like to thank Maria Eulalia Vares and an anonymous referee for comments which
improved the presentation of this work.

\bibliography{refs}

\providecommand{\bysame}{\leavevmode\hbox to3em{\hrulefill}\thinspace}
\providecommand{\MR}{\relax\ifhmode\unskip\space\fi MR }
\providecommand{\MRhref}[2]{%
  \href{http://www.ams.org/mathscinet-getitem?mr=#1}{#2}
}
\providecommand{\href}[2]{#2}
\begin{thebibliography}{CLMR10}

\bibitem[BOGP10]{basse-oconnor_martingale-type_2010}
Andreas Basse-O'Connor, Svend-Erik Graversen, and Jan Pedersen,
  \emph{Martingale-type processes indexed by the real line}, ALEA. Latin
  American Journal of Probability and Mathematical Statistics \textbf{7}
  (2010), 117--137, 00005 MR: 2651823.

\bibitem[BOGP14]{basse-oconnor_stochastic_2014}
A.~Basse-O'Connor, S.-E. Graversen, and J.~Pedersen, \emph{Stochastic
  integration on the real line}, Theory of Probability and its Applications
  \textbf{58} (2014), no.~2, 193--215. \MR{3300554}

\bibitem[Che08]{chen_limit_2008}
Xia Chen, \emph{Limit laws for the energy of a charged polymer}, Annales de
  l'Institut Henri Poincar{\'e} Probabilit{\'e}s et Statistiques \textbf{44} (2008),
  no.~4, 638--672, 00021 MR: 2446292.

\bibitem[CK91]{carlen_$l^p$_1991}
Eric Carlen and Paul Kr{\'e}e, \emph{$L^p$ estimates on
  iterated stochastic integrals}, The Annals of Probability \textbf{19} (1991),
  no.~1, 354--368, 00000.

\bibitem[CK09]{chen_charged_2009}
Xia Chen and Davar Khoshnevisan, \emph{From charged polymers to random walk in
  random scenery}, Optimality, {IMS} {Lecture} {Notes} {Monogr}. {Ser}.,
  vol.~57, Inst. Math. Statist., Beachwood, OH, 2009, 00016 MR: 2681685,
  pp.~237--251.

\bibitem[CLMR10]{chen_clt_2010}
Xia Chen, Wenbo~V. Li, Michael~B. Marcus, and Jay Rosen, \emph{A {CLT} for the
  $L^{2}$-modulus of continuity of {Brownian} local time},
  The Annals of Probability \textbf{38} (2010), no.~1, 396--438.

\bibitem[Doo90]{doob_stochastic_1990}
J.~L. Doob, \emph{Stochastic processes}, Wiley {Classics} {Library}, John Wiley
  \& Sons, Inc., New York, 1990, 00007 Reprint of the 1953 original, A
  Wiley-Interscience Publication. \MR{1038526}

\bibitem[HN09]{hu_stochastic_2009}
Yaozhong Hu and David Nualart, \emph{Stochastic integral representation of the
  $L^{2}$ modulus of {Brownian} local time and a central
  limit theorem}, Electronic Communications in Probability \textbf{14} (2009),
  529--539.

\bibitem[HN10]{hu_central_2010}
\bysame, \emph{Central limit theorem for the third moment in space of the
  {Brownian} local time increments}, Electronic Communications in Probability
  \textbf{15} (2010), 396--410.

\bibitem[Jac79]{jacod_calcul_1979}
Jean Jacod, \emph{Calcul stochastique et probl{\`e}mes de martingales}, Lecture
  {Notes} in {Mathematics}, vol. 714, Springer, Berlin, 1979.

\bibitem[Jeu85]{jeulin_application_1985}
T.~Jeulin, \emph{Application de la theorie du grossissement a l'etude des temps
  locaux {Browniens}}, Grossissements de filtrations: exemples et applications
  (Th.~Jeulin and M.~Yor, eds.), Lecture {Notes} in {Mathematics}, no. 1118,
  Springer Berlin Heidelberg, January 1985, Cited by 0000, pp.~197--304.

\bibitem[KS79]{kesten_limit_1979}
H.~Kesten and F.~Spitzer, \emph{A limit theorem related to a new class of
  self-similar processes}, Zeitschrift f{\"u}r Wahrscheinlichkeitstheorie und
  Verwandte Gebiete \textbf{50} (1979), no.~1, 5--25.

\bibitem[KS91]{karatzas_brownian_1991}
Ioannis Karatzas and Steven~E. Shreve, \emph{Brownian motion and stochastic
  calculus}, second ed., Graduate {Texts} in {Mathematics}, vol. 113,
  Springer-Verlag, New York, 1991, 09055 MR: 1121940.

\bibitem[MR08]{marcus_$l^p$_2008}
Michael~B. Marcus and Jay Rosen, \emph{$L^p$ moduli of
  continuity of {Gaussian} processes and local times of symmetric {L{\'e}vy}
  processes}, The Annals of Probability \textbf{36} (2008), no.~2, 594--622.

\bibitem[Per82]{perkins_local_1982}
Edwin Perkins, \emph{Local time is a semimartingale}, Zeitschrift f{\"u}r
  Wahrscheinlichkeitstheorie und Verwandte Gebiete \textbf{60} (1982), no.~1,
  79--117.

\bibitem[Pro04]{protter_stochastic_2004}
Philip~E. Protter, \emph{Stochastic integration and differential equations},
  second ed., Applications of {Mathematics} ({New} {York}), vol.~21,
  Springer-Verlag, Berlin, 2004, 02638 Stochastic Modelling and Applied
  Probability. \MR{2020294}

\bibitem[PY86]{pitman_asymptotic_1986}
Jim Pitman and Marc Yor, \emph{Asymptotic laws of planar {Brownian} motion},
  The Annals of Probability \textbf{14} (1986), no.~3, 733--779, 00144.

\bibitem[PY04]{peccati_hardys_2004}
Giovanni Peccati and Marc Yor, \emph{Hardy's inequality in
  $L^{2}([0,1])$ and principal values of {Brownian} local
  times}, Asymptotic methods in stochastics, Fields {Inst}. {Commun}., vol.~44,
  Amer. Math. Soc., Providence, RI, 2004, pp.~49--74.

\bibitem[Ros11a]{rosen_clt_2011}
Jay Rosen, \emph{A {CLT} for the third integrated moment of {Brownian} local
  time increments}, Stochastics and Dynamics \textbf{11} (2011), no.~1, 5--48.

\bibitem[Ros11b]{rosen_stochastic_2011}
\bysame, \emph{A stochastic calculus proof of the {CLT} for the
  $L^{2}$ modulus of continuity of local time}, S{\'e}minaire
  de {Probabilit{\'e}s} {XLIII}, Lecture {Notes} in {Math}., vol. 2006, Springer,
  Berlin, 2011, 00008, pp.~95--104.

\bibitem[Rud87]{rudin_real_1987}
Walter Rudin, \emph{Real and complex analysis}, third ed., McGraw-Hill Book
  Co., New York, 1987.

\bibitem[RW94]{rogers_diffusions_1994}
L.~C.~G. Rogers and David Williams, \emph{Diffusions, {Markov} processes, and
  martingales. {Vol}. 1}, second ed., Wiley {Series} in {Probability} and
  {Mathematical} {Statistics}: {Probability} and {Mathematical} {Statistics},
  John Wiley \& Sons, Ltd., Chichester, 1994, 00000 Foundations. \MR{1331599}

\bibitem[RW00]{rogers_diffusions_2000}
\bysame, \emph{Diffusions, {Markov} processes, and martingales. {Vol}. 2},
  Cambridge {Mathematical} {Library}, Cambridge University Press, Cambridge,
  2000, 00003 It{\^o} calculus, Reprint of the second (1994) edition. \MR{1780932}

\bibitem[RY99]{revuz_continuous_1999}
Daniel Revuz and Marc Yor, \emph{Continuous martingales and {Brownian} motion},
  third ed., Grundlehren der {Mathematischen} {Wissenschaften} [{Fundamental}
  {Principles} of {Mathematical} {Sciences}], vol. 293, Springer-Verlag,
  Berlin, 1999.

\bibitem[SK76]{segall_orthogonal_1976}
Adrian Segall and Thomas Kailath, \emph{Orthogonal functionals of
  independent-increment processes}, Institute of Electrical and Electronics
  Engineers. Transactions on Information Theory \textbf{IT-22} (1976), no.~3,
  287--298.

\bibitem[Sze75]{szegho_orthogonal_1975}
  G{\'a}bor Szeg\H{o}, \emph{Orthogonal polynomials}, fourth ed.,
  American Mathematical Society, Providence, R.I., 1975, American Mathematical
  Society, Colloquium Publications, Vol. XXIII.

\bibitem[Tak95]{takacs_brownian_1995}
Lajos Tak{\'a}cs, \emph{Brownian local times}, Journal of Applied Mathematics and
  Stochastic Analysis \textbf{8} (1995), no.~3, 209--232, 00010 MR: 1342642.

\bibitem[YL05]{yan_iterated_2005}
Litan Yan and Jingyun Ling, \emph{Iterated integrals with respect to {Bessel}
  processes}, Statistics \& Probability Letters \textbf{74} (2005), no.~1,
  93--102. \MR{2189721}

\end{thebibliography}
\end{document}